\documentclass[11pt]{article}

\usepackage[margin=1.12in]{geometry}
\usepackage{amsmath,amssymb,amsthm,mathtools}
\usepackage{bm}
\usepackage{enumitem}
\usepackage{microtype}
\usepackage[dvipsnames]{xcolor}
\usepackage[numbers,sort&compress]{natbib}
\usepackage[colorlinks=true,allcolors=MidnightBlue]{hyperref}
\usepackage[nameinlink,noabbrev]{cleveref}

\allowdisplaybreaks
\setlist[itemize]{leftmargin=1.7em}
\setlist[enumerate]{leftmargin=1.9em}

\newtheorem{theorem}{Theorem}[section]
\newtheorem{proposition}[theorem]{Proposition}
\newtheorem{lemma}[theorem]{Lemma}
\newtheorem{corollary}[theorem]{Corollary}
\newtheorem{assumption}[theorem]{Assumption}
\theoremstyle{remark}
\newtheorem{remark}[theorem]{Remark}

\newcommand{\T}{\mathbb T}
\newcommand{\R}{\mathbb R}
\newcommand{\Pp}{\mathcal P}
\newcommand{\F}{\mathcal F}
\newcommand{\E}{\mathbb E}
\newcommand{\Ent}{\operatorname{Ent}}
\newcommand{\diag}{\operatorname{diag}}
\newcommand{\dist}{\operatorname{dist}}
\newcommand{\dBL}{d_{\mathrm{BL}}}
\newcommand{\dd}{\,\mathrm d}
\newcommand{\bpi}{b_\pi}

\hypersetup{
  pdftitle={Riesz-Kernel Stein Variational Gradient Descent: Renormalized Entropy and Long-Time Particle Limits},
  pdfauthor={Trevor Teolis and Maarten V. de Hoop}
}

\title{\textbf{Riesz-Kernel Stein Variational Gradient Descent:}\\
Renormalized Entropy and Long-Time Particle Limits}
\author{Trevor Teolis\\Rice University
\and Maarten V. de Hoop\\Rice University}
\date{\today}

\begin{document}
\maketitle

\begin{abstract}
Stein variational gradient descent (SVGD) transports interacting particles
toward a target distribution through deterministic kernelized dynamics.
Singular Riesz kernels are attractive because they can provide quantitative
population-level convergence, but at the
finite-particle level the corresponding Stein energy has infinite
self-interaction.  We study periodic Riesz SVGD with self-interaction removed
and prove a many-particle, long-time sampling theorem.  Throughout the range
in which the singular Stein energy is locally integrable, under a uniform
bound on the initial relative entropy per particle, the time-averaged
empirical-measure law converges weakly to the point mass \(\delta_\pi\) at the
target as the particle number and any diverging averaging horizon tend to
infinity.  We also show that the empirical-measure laws induced by invariant
particle laws of finite relative entropy converge weakly to \(\delta_\pi\),
without a uniform entropy bound.  Below the logarithmic singularity threshold,
we obtain an explicit algebraic finite-particle error bound.  These results
extend the joint-entropy approach for smooth-kernel SVGD to singular
interactions.
\end{abstract}

\section{Introduction}

Sampling from a target probability density \(\pi\), known only up to
normalization, is a basic problem in Bayesian inference and computational
statistics.  At the population level, overdamped Langevin dynamics realizes
the Wasserstein
gradient flow of relative entropy through independent score-driven
diffusions.  The corresponding deterministic continuity equation has velocity
\(-\nabla\log(\rho/\pi)\), which depends on the unknown evolving density
\(\rho\) and therefore cannot be evaluated directly on an empirical measure.

Stein variational gradient descent (SVGD), introduced by
\citet{LiuWang2016}, takes a different route.  It projects this velocity onto
a kernel space and uses integration by parts to remove
\(\nabla\log\rho\).  Write \(\pi=Z^{-1}e^{-V}\) on a \(d\)-dimensional
state space, where \(Z\) is the normalizing constant, and let
\(\bpi:=\nabla\log\pi=-\nabla V\).  For a smooth symmetric positive-definite
kernel \(k\), the standard continuous-time particle system is
\begin{equation}\label{eq:standard-svgd-intro}
 \dot X_i(t)
 =\frac1N\sum_{j=1}^N
 \left\{\nabla_1 k(X_j(t),X_i(t))
       +k(X_j(t),X_i(t))\bpi(X_j(t))\right\}.
\end{equation}
Here \(\nabla_1\) differentiates the first kernel argument, and the sum
includes self-interaction.  The projected velocity depends on the target only
through its score, so the normalizing constant is not needed.  The
score-weighted term transports the ensemble toward the target, while the
kernel derivative supplies a repulsive interaction.  Thus SVGD moves the
empirical measure
\(\mu_t^N:=N^{-1}\sum_{i=1}^N\delta_{X_i(t)}\) without estimating its density
or choosing a parametric approximating family.  It trades independent
stochastic trajectories for a coordinated deterministic approximation of the
target law.

Our concern is the long-time behavior of this coordinated particle
transport.  Unlike Langevin dynamics, the finite system has no diffusive
mechanism enforcing ergodicity, and the particles cannot be analyzed as
independent samples.  The resulting question is how to obtain a
many-particle, long-time sampling statement from deterministic interacting
dynamics whose fixed-\(N\) trajectories need not converge to the target.
The quantity connecting the dynamics to convergence is the kernel Stein
discrepancy.  If \(\mathcal H_k\) is the reproducing-kernel Hilbert space of
\(k\), then
\begin{equation}\label{eq:standard-ksd-intro}
\begin{aligned}
 \operatorname{KSD}_k^2(\mu\mid\pi)
 &:=
 \left\|\int
 \bigl\{\nabla_x k(x,\cdot)+k(x,\cdot)\bpi(x)\bigr\}
 \,\dd\mu(x)\right\|_{\mathcal H_k^d}^{2}.
\end{aligned}
\end{equation}
It vanishes at the target and measures the size of the
kernel-projected transport velocity.  Along the population SVGD flow, it is
exactly the rate at which relative entropy decreases.  For \(P\ll Q\), write
\begin{equation}
\label{eq:relative-entropy}
\Ent(P\mid Q):=\int\log(\dd P/\dd Q)\,\dd P. 
\end{equation}
If \(\rho_t^k\) is the
smooth-kernel population SVGD flow, then
\begin{equation}\label{eq:standard-population-dissipation-intro}
 \frac{\dd}{\dd t}\Ent(\rho_t^k\mid\pi)
 =-\operatorname{KSD}_k^2(\rho_t^k\mid\pi).
\end{equation}
The population gradient-flow theory was developed by \citet{Liu2017} and
\citet{DuncanNuskenSzpruch2023}, while \citet{LuLuNolen2019} established the
mean-field limit.  For smooth kernels, \citet{CarrilloSkrzeczkowski2023}
subsequently obtained longer-time stability estimates for particle systems
initialized near the target.  The KSD is also used by
\citet{ChwialkowskiEtAl2016} to test whether a finite sample is consistent
with a specified target distribution and by \citet{GorhamMackey2017} to
quantify how well samples approximate the target.

This perspective has led both to new analyses and to modifications of the
algorithm.  \citet{WangTangBajajLiu2019} introduced matrix-valued kernels that
incorporate problem geometry through preconditioning, while
\citet{LiLiLiuLiuLu2020} used random batches to reduce the cost of the
all-to-all interaction.  Kernel design can also change the underlying
dissipation:
\citet{ChewiEtAl2020} recast SVGD as a kernelized Wasserstein gradient flow of
the chi-squared divergence and introduced Laplacian Adjusted Wasserstein
Gradient Descent (LAWGD), whose target-adapted kernel represents the inverse
Langevin generator and yields strong population convergence under a Poincar\'{e}
inequality.  The required spectral decomposition, however, makes LAWGD a
different computational regime from the translation-invariant scalar kernels
considered here.

For smooth kernels, the KSD can be evaluated directly on empirical measures,
including the self-interaction terms.  \citet{KorbaEtAl2020} established
non-asymptotic descent estimates for population SVGD and quantitative bounds
on its time-averaged KSD, \citet{NuskenRenger2023} developed a variational
description of the many-particle and long-time regimes, and
\citet{ShiMackey2023} obtained the first explicit finite-particle rate for
driving the KSD toward zero.  Most closely related to our argument,
\citet{BalasubramanianBanerjeeGhosal2024} introduced the normalized
joint-law entropy
\[
 H_N^{\mathrm{sm}}(t)
 :=\frac1N
 \Ent\!\left(P_t^{N,\mathrm{sm}}\mid\pi^{\otimes N}\right),
\]
where \(P_t^{N,\mathrm{sm}}\) is the joint law of the smooth-kernel particle
system.  Their entropy calculation yields
\[
 \frac1T\int_0^T
 \E\!\left[\operatorname{KSD}_k^2(\mu_t^N\mid\pi)\right]\,\dd t
 \le
 \frac{H_N^{\mathrm{sm}}(0)}{T}+\frac{C_k}{N}.
\]
Consequently, for uniformly bounded initial normalized entropy,
\[
 \lim_{T\to\infty}\limsup_{N\to\infty}
 \frac1T\int_0^T
 \E\!\left[\operatorname{KSD}_k^2(\mu_t^N\mid\pi)\right]\,\dd t
 =0.
\]
Under their additional assumptions, this also yields convergence of
time-averaged particle marginals.  More recently,
\citet{BalasubramanianBanerjeeKorba2026} proved propagation-of-chaos bounds,
uniform over the averaging horizon, for time-averaged smooth-kernel SVGD;
their uniform-in-physical-time parametric rates concern special finite-rank
systems.

Smooth-kernel results show that time-averaged many-particle convergence can
be obtained despite the absence of fixed-\(N\) ergodicity.  At the population
level, general results for standard smooth scalar kernels give qualitative
last-iterate convergence and quantitative time-averaged or best-iterate KSD
bounds, but no general quantitative rate is known for the last iterate in a
strong topology \citep{KorbaEtAl2020,ChizatColomboColomboFernandezReal2026}.
This contrast motivates kernels with finite-order Fourier decay.

On \(\T^d\), the mean-zero periodic Riesz kernel \(\kappa_a\) is
characterized by
\begin{equation}\label{eq:kappa-fourier}
 \widehat\kappa_a(0)=0,
 \qquad
 \widehat\kappa_a(\ell)=(2\pi|\ell|)^{-2a}
 \quad(\ell\in\mathbb Z^d\setminus\{0\}),
\end{equation}
or equivalently by \((-\Delta)^a\kappa_a=\delta_0-1\).  Its Fourier
multiplier makes the associated Stein discrepancy coercive in a negative
Sobolev norm, while the Riesz force provides short-range repulsion.
\citet{ChizatColomboColomboFernandezReal2026} show that this reduced smoothing
can yield quantitative population last-iterate convergence.  Under their
Sobolev regularity and small-initial-entropy assumptions, for every \(a>1\)
and a regularity index \(\gamma>\max\{d/2,a-1\}\),
\begin{equation}\label{eq:population-polynomial-rate}
 \Ent(\rho_t\mid\pi)
 \le
 \left(
 \Ent(\rho_0\mid\pi)^{-(a-1)/\gamma}+\frac{t}{C}
 \right)^{-\gamma/(a-1)},
\end{equation}
together with corresponding decay in strong Sobolev norms; in particular,
the representative entropy and squared \(L^2\) rates are
\(t^{-\gamma/(a-1)}\).  At the inverse-Laplacian endpoint \(a=1\),
\citet{CarrilloSkrzeczkowskiWarnett2024} establish global exponential
convergence for a class of kernels with quadratic Fourier decay, while
\citet{ChizatColomboColomboFernandezReal2026} recover the corresponding
Coulomb result on the torus.
In a different singular limit,
\citet{CarrilloSkrzeczkowskiWarnett2026} show that concentrating a regular
kernel leads to a local quadratic-mobility flow.

The same finite-order Fourier structure creates the particle-level
obstruction.  Applying the Stein operator in both variables produces a
scalar energy kernel with an infinite diagonal, whereas the singular particle
dynamics excludes self-interaction.  Thus the nonnegative population
dissipation cannot simply be evaluated on an empirical measure.  The issue is
not singularity by itself: reduced smoothing strengthens population
coercivity while simultaneously exposing the diagonal at finite \(N\).

\subsection{The particle system and its Stein energy}

We now take the state space to be the flat torus \(\T^d\) and specialize to
\(k(y,x)=\kappa_a(y-x)\).  The first-order Stein operator is
\begin{equation}\label{eq:A-intro}
 \mathcal A_{\pi,y}k(y,x)
 :=\nabla_y k(y,x)+\bpi(y)k(y,x).
\end{equation}
The Riesz interaction is not smooth on the diagonal, and its
self-interaction is not consistently defined across the range considered
below.  We therefore study the self-interaction-free system
\begin{equation}\label{eq:particle-system}
 \dot X_i
 =
 \frac1N\sum_{j\ne i}\mathcal A_{\pi,X_j}k(X_j,X_i)
 =
 \frac1N\sum_{j\ne i}
\left\{\nabla\kappa_a(X_j-X_i)
       +\kappa_a(X_j-X_i)\bpi(X_j)\right\}.
\end{equation}
Write \(\bm X(t):=(X_1(t),\ldots,X_N(t))\).  For a configuration
\(\bm x=(x_1,\ldots,x_N)\), let
\[
 \mu_{\bm x}^N:=\frac1N\sum_{i=1}^N\delta_{x_i}
\]
be its empirical measure, and let
\[
 \mathcal D_N:=\{\bm x\in(\T^d)^N:x_i\ne x_j\text{ for }i\ne j\}
\]
be the collision-free configuration space.
The first term is the repulsive kernel force; the second transports that
interaction through the target score.  With \(\kappa_a\) defined by
\eqref{eq:kappa-fourier}, we work in the finite-energy regime
\begin{equation}\label{eq:parameter-range-intro}
 1<a<1+\frac d2,
 \qquad
 \sigma:=d+2-2a\in(0,d).
\end{equation}

To connect the particle velocity with entropy dissipation, we use divergence
relative to the target measure.  For a vector field \(u\), this is
\[
 \mathcal S_{\pi,x}u
 :=\frac1{\pi(x)}\nabla_x\cdot\bigl(\pi(x)u\bigr)
 =\nabla_x\cdot u+\bpi(x)\cdot u.
\]
Applying this operator to the velocity generated at \(x\) by a source point
\(y\) defines the scalar Stein kernel
\begin{equation}\label{eq:G-intro}
 G_\pi(x,y):=\mathcal S_{\pi,x}\mathcal A_{\pi,y}k(y,x).
\end{equation}
Thus \(G_\pi(x,y)\) measures the contribution of the interaction between
\(x\) and \(y\) to compression of the flow relative to the target.  For
\(x\ne y\), a direct calculation gives
\begin{equation}\label{eq:G-expansion}
\begin{split}
 G_\pi(x,y)
 ={}&-\Delta\kappa_a(y-x)
 +\bigl(\bpi(x)-\bpi(y)\bigr)\cdot\nabla\kappa_a(y-x)\\
 &+\bpi(x)\cdot\bpi(y)\,\kappa_a(y-x).
\end{split}
\end{equation}

At the population level, the quadratic form generated by \(G_\pi\) is
nonnegative.  Indeed, for a smooth probability measure \(\mu\), 
\begin{equation}\label{eq:F-cont}
\begin{aligned}
 \operatorname{KSD}_{\kappa_a}^2(\mu\mid\pi)
 &:=
 \left\|
  \int_{\T^d}\mathcal A_{\pi,y}k(y,\cdot)\,\dd\mu(y)
 \right\|_{\mathcal H_{\kappa_a}^d}^{2}\\
 &=\iint_{\T^d\times\T^d}G_\pi(x,y)
   \,\dd(\mu-\pi)(x)\,\dd(\mu-\pi)(y)
 =:2\F^\pi(\mu).
\end{aligned}
\end{equation}
We call \(\F^\pi\) the Stein energy.  The space
\(\mathcal H_{\kappa_a}\) is the native Fourier space of \(\kappa_a\), defined
in \eqref{eq:native-space}.  Along the population SVGD flow, this energy is
exactly the entropy dissipation:
\begin{equation}\label{eq:population-dissipation-intro}
 \frac{\dd}{\dd t}\Ent(\rho_t\mid\pi)
 =-\operatorname{KSD}_{\kappa_a}^2(\rho_t\mid\pi)
 =-2\F^\pi(\rho_t).
\end{equation}
Identity \eqref{eq:F-cont} is proved for smooth measures in
\cref{prop:gram}, and the population entropy-dissipation identity
\eqref{eq:population-dissipation-intro} is proved in
\cref{app:population-dissipation}.  Both arguments use Fourier
duality and therefore do not require pointwise reproduction for the singular
kernel.

The role of \(G_\pi\) at the particle level is analogous.  When the Liouville
equation for the joint particle law is tested against relative entropy with
respect to \(\pi^{\otimes N}\), each ordered pair contributes
\(G_\pi(x_i,x_j)\).  The leading term
\(-\Delta\kappa_a(y-x)\) in \eqref{eq:G-expansion} diverges on the diagonal,
so the full empirical Stein energy is undefined.  Since the particle system
omits self-interaction, its entropy production instead involves the
off-diagonal empirical Stein energy
\begin{equation}\label{eq:F-discrete}
 \F_N^\pi(\bm x)
 :=\frac1{2N^2}\sum_{i\ne j}G_\pi(x_i,x_j).
\end{equation}

\subsection{The singular entropy argument}

The smooth-kernel entropy method suggests applying the same normalized
joint-law entropy to the off-diagonal Riesz system.  Let \(P_t^N\) be its
joint particle law and write
\[
 H_N(t):=\frac1N\Ent(P_t^N\mid\pi^{\otimes N}).
\]
Removing self-interaction gives the off-diagonal identity
\begin{equation}\label{eq:entropy-intro}
\begin{aligned}
 \frac{\dd}{\dd t}H_N(t)
 &=-\frac1{N^2}\E_{P_t^N}\!\left[\sum_{i\ne j}
     G_\pi(X_i(t),X_j(t))\right]\\
 &=-2\E_{P_t^N}\!\left[\F_N^\pi(\bm X(t))\right].
\end{aligned}
\end{equation}
The problem is to recover the nonnegativity lost by deleting an infinite
diagonal.

For comparison, modulated-energy methods compare an empirical measure with a
continuum density without evaluating the singular energy on the particle
diagonal.  The framework was developed for
Coulomb-type flows by \citet{Serfaty2020} and extended to translation-invariant
Riesz flows by \citet{NguyenRosenzweigSerfaty2022}, using renormalization and
smearing ideas related to \citet{PetracheSerfaty2017}.  The resulting
corrected modulated energy and commutator estimate give finite-time mean-field
convergence for Riesz flows.  In SVGD, however, the velocity decomposes as
\begin{equation}\label{eq:velocity-intro}
 v_\mu=-\nabla\kappa_a*\mu+\kappa_a*(\bpi\mu).
\end{equation}
The first term has the commutator structure treated by that theory within its
admissible range; the target-dependent second term creates a discrete
empirical trace for which the same estimate is unavailable.  This prevents a
direct application of existing Riesz-flow theory to a last-iterate particle
limit.  The time-averaged argument below avoids such a particle-to-population
comparison.

Our argument combines the joint-law entropy identity from smooth-kernel SVGD
with a direct capped-diagonal passage from particle energies to the continuum
KSD.  Making this work for the target-dependent, self-interaction-free Riesz
system requires three steps:

\begin{enumerate}
\item \emph{Singular entropy identity.}  We first prove global collision
avoidance and then justify the joint entropy calculation for the singular
flow.
\item \emph{Capped-diagonal empirical liminf.}  Capping the full Stein kernel
below its singular diagonal gives a direct lower bound from the off-diagonal
particle energy to the continuum KSD.  This proves that the negative part is
bounded by a correction \(c_N\to0\) and permits simultaneous averaging
horizons.
\item \emph{Quantitative correction below the logarithmic threshold.}  When
\(0<\sigma<2\), an operator factorization and a positive Bessel minorant give
the explicit rate \(c_N\lesssim N^{-1+\sigma/d}\).
\end{enumerate}

The resulting finite-\(N\) renormalized entropy method is the central
contribution: it converts the singular off-diagonal entropy production into a
nonnegative continuum discrepancy up to a vanishing \(N\)-dependent
correction.  Because the argument works directly with the joint particle law,
it does not require a propagation-of-chaos comparison with the population
flow.

\subsection{Main result and proof mechanism}

The exact off-diagonal identity
\eqref{eq:entropy-intro} does not yet provide dissipation because
\(\F_N^\pi\) need not be nonnegative.  Set
\begin{equation}\label{eq:onsager-intro}
 c_N:=\left(-\inf_{\bm x\in\mathcal D_N}\F_N^\pi(\bm x)\right)_+.
\end{equation}
The capped-diagonal liminf proves, throughout \(0<\sigma<d\), that
\(c_N\to0\).  Therefore the renormalized energy
\(\mathcal E_N^\pi:=\F_N^\pi+c_N\) is nonnegative.  Integrating
\eqref{eq:entropy-intro} and using \(H_N(T)\ge0\) gives
\begin{equation}\label{eq:rd-intro}
 \frac1T\int_0^T
 \E_{P_t^N}\!\left[\mathcal E_N^\pi(\bm X(t))\right]\,\dd t
 \le \frac{H_N(0)}{2T}+c_N.
\end{equation}
For \(0<\sigma<2\), the quantitative refinement gives
\begin{equation}\label{eq:quantitative-onsager-intro}
 c_N\le C_{\rm ons}N^{-1+\sigma/d}.
\end{equation}
Thus \eqref{eq:rd-intro} is the singular replacement for the finite-diagonal
correction in the smooth-kernel entropy method, with an explicit rate for
\(0<\sigma<2\).

The entropy identity is therefore naturally suited to time-averaged
convergence.  Let \(P_{t,i}^N\) be the \(i\)-th marginal of \(P_t^N\) and set
\begin{equation}\label{eq:tagged-average-intro}
 \overline P_{\mathrm{av},T}^{N}
 :=\frac1T\int_0^T\frac1N\sum_{i=1}^N P_{t,i}^N\,\dd t.
\end{equation}
Writing \(\dBL\) for the bounded-Lipschitz distance, we prove that whenever
\(T_N\to\infty\) and \(H_N(0)/T_N\to0\), the time-averaged empirical-measure
law converges to \(\delta_\pi\), the Dirac mass at the point
\(\pi\in\Pp(\T^d)\):
\begin{equation}\label{eq:ordered-result-intro}
 \frac1{T_N}\int_0^{T_N}
 (\bm x\mapsto\mu_{\bm x}^N)_\#P_t^N\,\dd t
 \rightharpoonup\delta_\pi,
 \qquad
 \dBL\!\left(\overline P_{\mathrm{av},T_N}^{N},\pi\right)\to0.
\end{equation}
The first convergence therefore takes place in \(\Pp(\Pp(\T^d))\).
In particular, every sequence \(T_N\to\infty\) is admissible under a uniform
initial normalized-entropy bound.  No exchangeability is required.  We also
prove that invariant laws of finite relative entropy have empirical-measure
laws converging to \(\delta_\pi\), without any uniform entropy bound.
Exchangeability is needed only to conclude that their first marginals converge
to \(\pi\).

The proof of \eqref{eq:ordered-result-intro} has four conceptual steps.
First, the Riesz repulsion prevents collisions, so the singular flow and its
Liouville equation are globally defined.  Second, the joint-entropy identity
gives an upper bound on the averaged off-diagonal energy.  Third, a bounded
continuous cap of \(G_\pi\) gives a lower bound for every sequence of
empirical-measure laws, including laws averaged over \(T_N\).  Finally,
compactness and the coercive identity \(\F^\pi(\mu)=0\) if and only if
\(\mu=\pi\) identify the limit.  The Bessel-minorant argument is an additional
quantitative refinement and is not needed for \eqref{eq:ordered-result-intro}.

\section{Setting and theorem statements}\label{sec:setting}

\subsection{The torus, the target, and the functional setting}

We identify \(\T^d\) with \([-\tfrac12,\tfrac12)^d\) and use normalized
Lebesgue measure.  Fourier coefficients are
\[
 \widehat f(\ell)=\int_{\T^d}f(x)e^{-2\pi i\ell\cdot x}\,\dd x.
\]
We use the bounded-Lipschitz metric
\[
 \dBL(\mu,\nu)
 :=
 \sup_{\|f\|_\infty+\operatorname{Lip}(f)\le1}
 \left|\int_{\T^d}f\,\dd(\mu-\nu)\right|,
\]
which metrizes weak convergence on \(\Pp(\T^d)\).

\begin{assumption}[Target]\label{ass:target}
The density \(\pi=Z^{-1}e^{-V}\) is strictly positive.  For some
\[
 m>\frac d2+2,
 \qquad m\ge a+1,
\]
we have \(\pi,V\in H^m(\T^d)\) and \(V\in W^{2,\infty}(\T^d)\).
\end{assumption}

The Sobolev assumptions in \cref{ass:target} are used in the continuum
identification of the Riesz KSD with a negative Sobolev norm and in the
operator estimates for \(0<\sigma<2\) in
\cref{thm:quantitative-onsager}.  Collision avoidance and the entropy
calculation use only the bounded derivatives of the target score displayed in
their proofs.

The Riesz kernel \(\kappa_a\) was defined in \eqref{eq:kappa-fourier}.  Its
native space is the mean-zero Fourier space
\begin{equation}\label{eq:native-space}
 \mathcal H_{\kappa_a}
 :=\left\{f:\widehat f(0)=0,\ 
 \|f\|_{\mathcal H_{\kappa_a}}^2
 :=\sum_{\ell\ne0}\widehat\kappa_a(\ell)^{-1}
 |\widehat f(\ell)|^2<\infty\right\}
 =\dot H^a_0(\T^d).
\end{equation}
If \(a>d/2\), point evaluation is continuous and this is the usual
mean-zero reproducing-kernel Hilbert space generated by \(\kappa_a\).  The
range considered here may also contain \(a\le d/2\), where point evaluation
on \(\dot H^a\) is not continuous and \(\kappa_a\) is singular on the
diagonal.  The literal pointwise reproducing property is then unavailable.
We therefore use \eqref{eq:native-space} through Fourier duality; no point
evaluation of the singular kernel is needed.

The main finite-particle theorems assume
\eqref{eq:parameter-range-intro}.  In this range the kernel is smooth away
from the origin, its force is less singular than the Coulomb force
\((a=1,\sigma=d)\), and the
scalar kernel \(G_\pi\) has locally integrable singularity
\(\dist(x,y)^{-\sigma}\).

The identity \eqref{eq:F-cont} initially defines \(\F^\pi\) for smooth
measures.  Later we apply it to weak limits of empirical measures, which need
not have densities.  We therefore use \(\F^\pi\) for the
lower-semicontinuous relaxation of the smooth Gram energy to
\(\Pp(\T^d)\), allowing the value \(+\infty\).  The diagonal-cap
construction in \cref{sec:liminf} gives a concrete representation of this
relaxation.

\subsection{Entropy dissipation and long-time sampling}

Recall from the introduction that \(\mathcal D_N\) denotes the collision-free
configuration space on which the particle vector field is defined.

The first theorem establishes the analytic foundation for the entropy method.
Its three conclusions have distinct roles: collision avoidance makes the
singular flow global; the Liouville calculation identifies its exact entropy
production; and a qualitative correction turns that sign-indefinite
off-diagonal production into a nonnegative energy up to a vanishing
finite-\(N\) error.

\begin{theorem}[Global dynamics and renormalized entropy]\label{thm:entropy}
Suppose \cref{ass:target} and \eqref{eq:parameter-range-intro} hold.

\begin{enumerate}
\item Every solution of \eqref{eq:particle-system} starting in
\(\mathcal D_N\) exists globally and remains in \(\mathcal D_N\).
\item Let \(P_0^N\ll\pi^{\otimes N}\) have finite relative entropy, and let
\(P_t^N\) be its pushforward by the flow.  Then, for every \(t\ge0\),
\begin{equation}\label{eq:entropy-identity}
 H_N(t)+2\int_0^t\E_{P_s^N}\F_N^\pi\,\dd s=H_N(0),
 \qquad
 H_N(t):=\frac1N\Ent(P_t^N\mid\pi^{\otimes N}).
\end{equation}
\item If
\begin{equation}\label{eq:qualitative-correction}
 m_N:=\inf_{\bm x\in\mathcal D_N}\F_N^\pi(\bm x),
 \qquad c_N:=(-m_N)_+,
 \qquad \mathcal E_N^\pi:=\F_N^\pi+c_N,
\end{equation}
then \(c_N\to0\), \(\mathcal E_N^\pi\ge0\) on \(\mathcal D_N\), and
\begin{equation}\label{eq:renormalized-dissipation}
 \frac1T\int_0^T\E_{P_t^N}\mathcal E_N^\pi\,\dd t
 \le \frac{H_N(0)}{2T}+c_N
 \qquad(T>0).
\end{equation}
\end{enumerate}
\end{theorem}

For a smooth kernel, the analogue of the third conclusion follows by adding
the finite diagonal of the Gram matrix.  Here that diagonal is infinite.
The correction in \eqref{eq:qualitative-correction} is shown to vanish by the
direct capped-diagonal liminf of \cref{prop:empirical-liminf}; no smearing
estimate is needed for this qualitative conclusion.

Below the logarithmic threshold \(\sigma=2\), the correction has an explicit
natural-scale bound.

\begin{theorem}[Quantitative Onsager bound for \(0<\sigma<2\)]
\label{thm:quantitative-onsager}
Under the assumptions of \cref{thm:entropy}, assume in addition that
\(0<\sigma<2\).  There is \(C_{\rm ons}<\infty\), depending only on
\(d,a,\pi\), such that
\begin{equation}\label{eq:onsager}
 \F_N^\pi(\bm x)\ge-C_{\rm ons}N^{-1+\sigma/d}
 \qquad(N\ge2,\ \bm x\in\mathcal D_N).
\end{equation}
Consequently,
\begin{equation}\label{eq:quantitative-correction}
 c_N\le C_{\rm ons}N^{-1+\sigma/d}.
\end{equation}
Moreover, for the particle laws in \cref{thm:entropy},
\begin{equation}\label{eq:quantitative-dissipation}
 \frac1T\int_0^T\E_{P_t^N}\!\left[
  \F_N^\pi+C_{\rm ons}N^{-1+\sigma/d}\right]\,\dd t
 \le \frac{H_N(0)}{2T}+C_{\rm ons}N^{-1+\sigma/d}.
\end{equation}
\end{theorem}

The next theorem converts the averaged energy bound into concentration of the
time-averaged empirical-measure law at the target, and hence convergence of
its label-averaged marginal.

\begin{theorem}[Simultaneous many-particle and long-time limit]
\label{thm:ordered}
Under the assumptions of \cref{thm:entropy}, let \(T_N\to\infty\) and assume
\begin{equation}\label{eq:entropy-time-condition}
 \frac{H_N(0)}{T_N}\longrightarrow0.
\end{equation}
Define the time-averaged empirical-measure law
\begin{equation}\label{eq:Lambda-simultaneous}
 \Lambda_{N,T_N}
 :=\frac1{T_N}\int_0^{T_N}
 (\bm x\mapsto\mu_{\bm x}^N)_\#P_t^N\,\dd t.
\end{equation}
Then
\begin{equation}\label{eq:ordered-main}
 \Lambda_{N,T_N}\rightharpoonup\delta_\pi,
 \qquad
 \dBL\!\left(\overline P_{\mathrm{av},T_N}^{N},\pi\right)
 \longrightarrow0.
\end{equation}
In particular, \eqref{eq:ordered-main} holds for every \(T_N\to\infty\) if
\(\sup_NH_N(0)<\infty\).
\end{theorem}

The theorem concerns the full law of the empirical measure and therefore also
the label-averaged marginal.  It does not require the joint law to be
exchangeable.  Condition \eqref{eq:entropy-time-condition} shows that a
uniform entropy bound is sufficient but not necessary.

At stationarity, the entropy identity forces the mean off-diagonal
dissipation to vanish.  Passing this fact to the continuum yields the
following consequence.

\begin{corollary}[No macroscopic spurious stationary law]\label{cor:stationary}
For each \(N\), let \(Q^N\ll\pi^{\otimes N}\) be invariant under the flow
\eqref{eq:particle-system} and have finite relative entropy.  Then
\begin{equation}\label{eq:stationary-empirical-law}
 (\bm x\mapsto\mu_{\bm x}^N)_\#Q^N\rightharpoonup\delta_\pi,
 \qquad
 \frac1N\sum_{i=1}^NQ_i^N\rightharpoonup\pi.
\end{equation}
Here \(Q_i^N\) denotes the \(i\)-th marginal of \(Q^N\).  If the laws are
exchangeable, then \(Q_1^N\rightharpoonup\pi\).
\end{corollary}

\section{Stein--Riesz structure}\label{sec:structure}

This section establishes the two properties of \(G_\pi\) used throughout the
proof.  Locally, it has a positive Riesz singularity; this drives collision
avoidance and permits a monotone cap of the diagonal.  Globally, it is a centered
Gram kernel; this makes the continuum energy nonnegative and identifies it as
a KSD.  The finite-particle difficulty is precisely that these global
properties do not make sense on the infinite diagonal.

\subsection{Local singularity}

\begin{proposition}[Local expansion]\label{prop:local-expansion}
There are \(r_0,c_0,C>0\) such that, for
\(0<r:=\dist(x,y)<r_0\),
\begin{equation}\label{eq:local-G}
 G_\pi(x,y)=c_0r^{-\sigma}+R_\pi(x,y),
\end{equation}
where
\begin{equation}\label{eq:R-bound}
 |R_\pi(x,y)|
 \le C
 \begin{cases}
 1,&0<\sigma<2,\\
 1+|\log r|,&\sigma=2,\\
 1+r^{-(\sigma-2)},&2<\sigma<d.
 \end{cases}
\end{equation}
In particular, after decreasing \(r_0\) if needed,
\begin{equation}\label{eq:G-positive-near}
 G_\pi(x,y)\ge\frac{c_0}{2}\dist(x,y)^{-\sigma}>0
 \qquad(0<\dist(x,y)<r_0).
\end{equation}
\end{proposition}

\begin{proof}
The periodic Riesz asymptotics in \cref{lem:riesz-asymptotics} give
\[
 -\Delta\kappa_a(z)=c_0|z|^{-\sigma}+O(1)
\]
and
\[
 |\kappa_a(z)|\lesssim
 \begin{cases}
 1,&\sigma<2,\\
 1+|\log|z||,&\sigma=2,\\
 |z|^{-(\sigma-2)},&\sigma>2,
 \end{cases}
 \qquad
 |\nabla\kappa_a(z)|\lesssim 1+|z|^{1-\sigma}.
\]
Since \(\bpi\) is Lipschitz,
\[
 |(\bpi(x)-\bpi(y))\cdot\nabla\kappa_a(y-x)|
 \lesssim |x-y|\bigl(1+|x-y|^{1-\sigma}\bigr).
\]
Substitution in \eqref{eq:G-expansion} proves
\eqref{eq:local-G}--\eqref{eq:R-bound}.  Every remainder is lower order than
\(r^{-\sigma}\), which yields \eqref{eq:G-positive-near}.
\end{proof}

The proposition shows that the target dependence does not alter the leading
singularity.  The positivity in \eqref{eq:G-positive-near} will be used twice:
first to repel collapsing clusters, and later to cap the singular diagonal
from below without losing an empirical lower bound.

\subsection{Stein centering and positivity}

Local positivity controls the singularity but does not yet explain why the
continuum energy measures discrepancy from \(\pi\).  That information comes
from the global Stein identities.

\begin{proposition}[Centering and Gram representation]\label{prop:gram}
The kernel is centered in each variable:
\begin{equation}\label{eq:centering}
 \int_{\T^d}G_\pi(x,y)\,\dd\pi(y)=0,
 \qquad
 \int_{\T^d}G_\pi(x,y)\,\dd\pi(x)=0
\end{equation}
in the distributional sense.  If \(\mu\) is smooth, then
\begin{equation}\label{eq:gram}
 2\F^\pi(\mu)
 =
 \left\|
 \int_{\T^d}\mathcal A_{\pi,y}k(y,\cdot)\,\dd\mu(y)
 \right\|_{\mathcal H_{\kappa_a}^d}^{\,2}\ge0.
\end{equation}
Consequently \(\F^\pi\) is convex.
\end{proposition}

\begin{proof}
We first record the two integrations by parts separately.  For every smooth
periodic vector field \(u\),
\[
 \int_{\T^d}\mathcal S_{\pi}u\,\dd\pi
 =\int_{\T^d}\nabla\cdot(u\pi)\,\dd x=0.
\]
For fixed \(x\), the feature is also centered:
\[
 \int_{\T^d}\mathcal A_{\pi,y}k(y,x)\,\dd\pi(y)
 =\int_{\T^d}\nabla_y\!\bigl(k(y,x)\pi(y)\bigr)\,\dd y=0.
\]
These formulas are understood after testing in the remaining variable, so
they remain valid for the singular kernel.  Apply the first formula to
\(u(x)=\mathcal A_{\pi,y}k(y,x)\), and apply the second before
\(\mathcal S_{\pi,x}\), to obtain the two identities in
\eqref{eq:centering}.

We next make the Gram calculation explicit without invoking pointwise
reproduction.  Write \(\dd\mu=\rho\,\dd x\) and set
\[
 q_\rho:=\bpi\rho-\nabla\rho.
\]
Integration by parts in the feature variable gives
\[
 \Phi_\rho
 :=\int_{\T^d}\mathcal A_{\pi,y}k(y,\cdot)\rho(y)\,\dd y
 =\kappa_a*q_\rho.
\]
For each fixed \(y\), integration by parts in \(x\) gives
\[
 \int_{\T^d}\rho(x)\,
 \mathcal S_{\pi,x}\mathcal A_{\pi,y}k(y,x)\,\dd x
 =\int_{\T^d}q_\rho(x)\cdot
 \mathcal A_{\pi,y}k(y,x)\,\dd x.
\]
Integrating against \(\rho(y)\,\dd y\) and using the definition of
\(\Phi_\rho\), we find
\begin{equation}\label{eq:gram-intermediate}
 \iint G_\pi(x,y)\rho(x)\rho(y)\,\dd x\,\dd y
 =\int_{\T^d}q_\rho(x)\cdot\Phi_\rho(x)\,\dd x
 =\int_{\T^d}q_\rho\cdot(\kappa_a*q_\rho)\,\dd x.
\end{equation}
This is the quadratic form associated with the scalar Stein kernel.
Since \(\widehat\kappa_a(\ell)>0\) for \(\ell\ne0\), Parseval's identity
yields
\[
 \int q_\rho\cdot(\kappa_a*q_\rho)
 =\sum_{\ell\ne0}\widehat\kappa_a(\ell)
   |\widehat q_\rho(\ell)|^2
 =\sum_{\ell\ne0}\widehat\kappa_a(\ell)^{-1}
   |\widehat\Phi_\rho(\ell)|^2
 =\|\Phi_\rho\|_{\mathcal H_{\kappa_a}^d}^{2},
\]
where the middle equality uses
\(\widehat\Phi_\rho=\widehat\kappa_a\widehat q_\rho\), and the last is
exactly \eqref{eq:native-space}.  Finally, the feature generated by \(\pi\)
vanishes because
\(\bpi\pi-\nabla\pi=0\).  Thus the same identity holds with
\(\rho\) replaced by \(\rho-\pi\).  Together with
\eqref{eq:centering} and \eqref{eq:gram-intermediate}, this proves
\eqref{eq:gram}.  Convexity follows because the right-hand side is the
squared norm of a linear function of \(\mu\); it passes to the
lower-semicontinuous extension by approximation.
\end{proof}

\begin{remark}
Equations \eqref{eq:centering} and \eqref{eq:gram} explain why the centered
form in \eqref{eq:F-cont} equals
\(\frac12\iint G_\pi\,\dd\mu\,\dd\mu\) whenever either expression is finite.
They do not justify inserting an atom on the diagonal; that is precisely the
finite-\(N\) renormalization problem.
\end{remark}

\section{Collision-free particle dynamics and Liouville structure}
\label{sec:dynamics}

Before differentiating the joint entropy, we must know that the singular ODE
is globally well-posed.  The only possible finite-time obstruction is a
collision.  Below the logarithmic threshold, the radius of a collapsing
cluster has a strictly positive derivative at each sufficiently small running
minimum.  At and above the threshold, the Riesz pair energy diverges at
collision and is decreasing near the collision set.  Once collisions are
excluded, the Liouville calculation can be performed on the collision-free
configuration space.

\subsection{Cluster repulsion below the logarithmic threshold}

\begin{lemma}[Cluster differential inequality]\label{lem:cluster}
Assume \(0<\sigma<2\).  In local Euclidean lifts, let
\[
 \bar X_I=\frac1{|I|}\sum_{i\in I}X_i,
 \qquad
 R_I^2=\sum_{i\in I}|X_i-\bar X_I|^2
\]
for a set \(I\) with \(|I|\ge2\).  If, at the time under consideration,
\(R_I\) is sufficiently small and the particles in \(I\) are separated from
those outside \(I\) by a fixed positive distance, then
\begin{equation}\label{eq:cluster-ineq}
 \frac12\frac{\dd}{\dd t}R_I^2
 \ge c\bigl(R_I^2\bigr)^{1-\sigma/2}-CR_I^2,
\end{equation}
where \(c,C>0\) may depend on the separation and on \(I\).
\end{lemma}

\begin{proof}
Write \(q_i=X_i-\bar X_I\).  Symmetrizing the internal Riesz force gives
\[
 \frac1N\sum_{i\in I}q_i\cdot
 \sum_{\substack{j\in I\\j\ne i}}\nabla\kappa_a(X_j-X_i)
 =
 \frac1N\sum_{\substack{i<j\\i,j\in I}}
 -(X_j-X_i)\cdot\nabla\kappa_a(X_j-X_i).
\]
The asymptotic bound
\begin{equation}\label{eq:radial-repulsion}
 -z\cdot\nabla\kappa_a(z)\ge c|z|^{2-\sigma}
\end{equation}
and the identity
\(\sum_{i<j}|X_i-X_j|^2=|I|R_I^2\) show that this is bounded below by
\(c(R_I^2)^{1-\sigma/2}\).

For the target-weighted internal force, subtract
\((|I|-1)\kappa_a(0)\bpi(\bar X_I)\) from each inner sum; its total
contribution vanishes because \(\sum_iq_i=0\).  The remaining summands are
\[
 \bigl(\kappa_a(X_j-X_i)-\kappa_a(0)\bigr)\bpi(X_j)
 +\kappa_a(0)\bigl(\bpi(X_j)-\bpi(\bar X_I)\bigr).
\]
Since
\[
 |\kappa_a(z)-\kappa_a(0)|\lesssim |z|^{2-\sigma},
 \qquad
 |\bpi(X_j)-\bpi(\bar X_I)|\lesssim R_I,
\]
their contribution is \(O(R_I^{3-\sigma})+O(R_I^2)\), which is lower order
than \(R_I^{2-\sigma}\).  Finally, the velocity generated by particles
outside \(I\) is Lipschitz across the separated cluster.  Subtracting its
value at \(\bar X_I\) leaves \(O(R_I^2)\).  These estimates give
\eqref{eq:cluster-ineq}.
\end{proof}

\subsection{A pair-energy barrier at and above the logarithmic threshold}

Define
\begin{equation}\label{eq:pair-energy}
 \mathcal W_N(\bm x):=\sum_{1\le i<j\le N}\kappa_a(x_j-x_i)
\end{equation}
and let \(\bm C^N=(C_1^N,\ldots,C_N^N)\), where
\[
 C_i^N(\bm x)
 :=\frac1N\sum_{j\ne i}\kappa_a(x_j-x_i)\bpi(x_j).
\]
Then the particle system can be written as
\begin{equation}\label{eq:pair-energy-flow}
 \dot{\bm X}=-\frac1N\nabla\mathcal W_N(\bm X)+\bm C^N(\bm X).
\end{equation}

\begin{lemma}[Pair-energy barrier]\label{lem:pair-energy-barrier}
Assume \(2\le\sigma<d\).  As a collision-free configuration approaches the
collision set,
\begin{equation}\label{eq:pair-barrier-limits}
 \mathcal W_N(\bm x)\longrightarrow+\infty,
 \qquad
 \frac{|\bm C^N(\bm x)|}{|\nabla\mathcal W_N(\bm x)|}
 \longrightarrow0.
\end{equation}
\end{lemma}

\begin{proof}
Let \(r(\bm x):=\min_{i\ne j}\dist(x_i,x_j)\).  The first limit follows from
the logarithmic divergence of \(\kappa_a\) when \(\sigma=2\), its
\(r^{2-\sigma}\) divergence when \(\sigma>2\), and its boundedness from below
away from the diagonal.  Moreover,
\[
 |\bm C^N(\bm x)|
 \le C
 \begin{cases}
  1+|\log r(\bm x)|,&\sigma=2,\\
  r(\bm x)^{2-\sigma},&\sigma>2.
 \end{cases}
\]
The closest-scale cluster estimate in
\cref{lem:pair-energy-gradient} gives
\[
 |\nabla\mathcal W_N(\bm x)|
 \ge c\,r(\bm x)^{1-\sigma}
\]
for all sufficiently small \(r(\bm x)\).  The quotient is therefore bounded
by \(Cr(1+|\log r|)\) at \(\sigma=2\) and by \(Cr\) when \(\sigma>2\), proving
\eqref{eq:pair-barrier-limits}.
\end{proof}

\begin{proposition}[No collision]\label{prop:no-collision}
Every solution of \eqref{eq:particle-system} with initial condition in
\(\mathcal D_N\) is global and collision-free.
\end{proposition}

\begin{proof}
The vector field is locally Lipschitz on \(\mathcal D_N\), so a maximal
solution can fail to continue only by approaching the collision set.  Suppose
first that \(0<\sigma<2\) and a first collision occurs at
\(\tau<\infty\).  For \(|I|\ge2\), set
\[
 Y_I(t):=\frac1{|I|}\sum_{\substack{i<j\\i,j\in I}}
 \dist(X_i(t),X_j(t))^2.
\]
On a sufficiently small sublevel, the particles in \(I\) lie in a common
coordinate chart and \(Y_I=R_I^2\).  A collision sequence supplies at least
one set with \(\liminf_{t\uparrow\tau}Y_I(t)=0\); among all such sets, choose
one that is maximal under inclusion.  The elementary running-minimum
selection then gives times \(t_n\uparrow\tau\) such that
\[
 Y_I(t_n)\longrightarrow0,
 \qquad
 Y_I'(t_n)\le0.
\]
Maximality implies, after passing to a subsequence, that the particles in
\(I\) remain a fixed positive distance from every exterior particle at
these times; otherwise \(I\) could be enlarged.  Applying
\cref{lem:cluster} at \(t_n\) gives a strictly positive derivative for all
large \(n\), a contradiction.

Now let \(2\le\sigma<d\).  By \cref{lem:pair-energy-barrier}, sufficiently
near the collision set,
\[
 |\bm C^N|\le\frac1{2N}|\nabla\mathcal W_N|.
\]
Along \eqref{eq:pair-energy-flow},
\[
 \frac{\dd}{\dd t}\mathcal W_N(\bm X(t))
 =-\frac1N|\nabla\mathcal W_N|^2
   +\nabla\mathcal W_N\cdot\bm C^N
 \le-\frac1{2N}|\nabla\mathcal W_N|^2.
\]
To account for trajectories that enter and leave this neighborhood, choose
\(\delta>0\) so that the displayed inequality holds whenever
\(r(\bm X)<\delta\).  On every connected component of
\(\{t:r(\bm X(t))<\delta\}\), the pair energy is nonincreasing.  At an entry
time it is uniformly bounded because \(\{\bm x:r(\bm x)\ge\delta\}\) is
compact; on a component containing the initial time it is bounded by its
initial value.  Thus \(\mathcal W_N\) stays uniformly bounded along every
collision sequence, contradicting \eqref{eq:pair-barrier-limits}.  Hence no
finite-time collision occurs in either range.
\end{proof}

Having ruled out collisions, we pass from individual trajectories to the law
of the full particle configuration.  The appropriate equation is the
Liouville equation, the continuity equation for a probability density
transported by a deterministic vector field.  This is useful here because
differentiating relative entropy along a continuity equation produces the
divergence of that vector field relative to the reference measure
\(\pi^{\otimes N}\).  The next identity recognizes this relative divergence
as the off-diagonal Stein energy.

\subsection{Liouville equation}

From this point on, denote the right-hand side of
\eqref{eq:particle-system} by
\[
 B_i^N(\bm X)
 :=\frac1N\sum_{j\ne i}\mathcal A_{\pi,X_j}k(X_j,X_i),
 \qquad
 \dot X_i=B_i^N(\bm X).
\]

\begin{lemma}[Divergence identity]\label{lem:divergence}
On \(\mathcal D_N\), the vector field
\(\bm B^N=(B_1^N,\ldots,B_N^N)\) satisfies
\begin{equation}\label{eq:divergence}
 \operatorname{div}_{\bm x}\bm B^N(\bm x)
 +\sum_{i=1}^N\bpi(x_i)\cdot B_i^N(\bm x)
 =2N\F_N^\pi(\bm x).
\end{equation}
\end{lemma}

\begin{proof}
By \eqref{eq:G-intro},
\[
 \nabla_{x_i}\cdot\mathcal A_{\pi,x_j}k(x_j,x_i)
 +\bpi(x_i)\cdot\mathcal A_{\pi,x_j}k(x_j,x_i)
 =G_\pi(x_i,x_j).
\]
Sum over \(i\ne j\) and divide by \(N\).
\end{proof}

Thus \(G_\pi\) is exactly the pairwise contribution to the compressibility of
the joint flow relative to \(\pi^{\otimes N}\).  Identity
\eqref{eq:divergence} is what
turns the Liouville transport equation into the entropy-dissipation identity
in \cref{prop:entropy-identity}.

If \(P_t^N\) has density \(p_t^N\), it solves
\begin{equation}\label{eq:liouville}
 \partial_t p_t^N+\operatorname{div}_{\bm x}(p_t^N\bm B^N)=0
\end{equation}
on \(\mathcal D_N\).  No boundary condition is imposed at collisions: by
\cref{prop:no-collision}, the flow never reaches that boundary.

\section{Joint entropy and Onsager renormalization}\label{sec:entropy}

We now turn the deterministic flow into a dissipative estimate for its joint
law.  The calculation has two parts.  The Liouville equation gives an exact
identity with \(\F_N^\pi\) as entropy production.  Since this off-diagonal
quantity can be negative, the direct empirical liminf in
\cref{prop:empirical-liminf} supplies a qualitative vanishing correction.
For \(0<\sigma<2\), the Bessel-minorant argument of
\cref{sec:quantitative-onsager} gives its explicit rate.

\subsection{Exact entropy production}

The entropy identity below is the structural ingredient inherited from the
smooth-kernel joint-law method.  The work here is to justify it despite the
singular collision set; the no-collision theorem and the approximation in
\cref{app:entropy-approx} provide that justification.

\begin{proposition}[Entropy identity]\label{prop:entropy-identity}
Let \(P_0^N\ll\pi^{\otimes N}\) have finite relative entropy, and let
\(P_t^N=(\Phi_t^N)_\#P_0^N\), where \(\Phi_t^N\) is the collision-free flow of
\eqref{eq:particle-system}.  Then \eqref{eq:entropy-identity} holds for every
\(t\ge0\).
\end{proposition}

\begin{proof}
Assume first that \(p_t^N\) is smooth and that its support stays a positive
distance from the collision set on the time interval under consideration.
Mass conservation removes the derivative of the linear part of the entropy,
and integration by parts in \eqref{eq:liouville} gives
\begin{align*}
\frac{\dd}{\dd t}\Ent(P_t^N\mid\pi^{\otimes N})
&=\int p_t^N\bm B^N\cdot
 \nabla_{\bm x}\log\frac{p_t^N}{\pi^{\otimes N}}\,\dd\bm x\\
&=-\int p_t^N\left(
 \operatorname{div}_{\bm x}\bm B^N+
 \sum_i\bpi(x_i)\cdot B_i^N\right)\dd\bm x\\
&=-2N\,\E_{P_t^N}\F_N^\pi,
\end{align*}
where the last equality is \eqref{eq:divergence}.  Divide by \(N\) and
integrate in time.

For a general finite-entropy initial law, localize to sets on which the
density ratio is bounded and trajectories remain uniformly separated from
collisions.  The same calculation applies on each localized flow tube.
The exhaustion argument, including passage of the shifted energy and both
endpoint entropies to the limit, is given in \cref{app:entropy-approx}.
\end{proof}

\subsection{Renormalized dissipation}

The entropy identity alone is insufficient because its production term is
the off-diagonal energy, not the nonnegative continuum KSD.  The empirical
liminf in \cref{sec:liminf} supplies the vanishing correction throughout
\(0<\sigma<d\); \cref{sec:quantitative-onsager} later gives its explicit rate
for \(0<\sigma<2\).

\begin{proof}[Proof of \cref{thm:entropy}]
Global existence is \cref{prop:no-collision}, and the exact identity is
\cref{prop:entropy-identity}.  The qualitative correction
\(c_N\to0\) and \(\mathcal E_N^\pi\ge0\) are proved in
\cref{prop:qualitative-correction}.  Since relative entropy is nonnegative,
\eqref{eq:entropy-identity} then gives
\[
 \frac1T\int_0^T\E\mathcal E_N^\pi\,\dd t
 =\frac{H_N(0)-H_N(T)}{2T}+c_N
 \le\frac{H_N(0)}{2T}+c_N.
\]
\end{proof}

\section{Passage to continuum KSD}\label{sec:liminf}

Estimate \eqref{eq:renormalized-dissipation} controls a scalar particle
energy.  To obtain convergence of measures, we must show that this energy
cannot disappear in the many-particle limit.  We consider laws of the
empirical measure \(\mu_{\bm x}^N\) defined in the introduction.
Compactness of \(\Pp(\T^d)\) gives subsequential limits automatically.  The
work is to pass the singular off-diagonal energy to those limits.  We do this
by capping \(G_\pi\) at a finite height, passing the resulting bounded
continuous kernel to the limit, and then raising the cap.  The finite
diagonal introduced by the cap costs only \(L/(2N)\).  Since this comparison
holds for every configuration law, it also applies when the averaging
horizon depends on \(N\).

\subsection{A capped-diagonal representation}

By \cref{prop:local-expansion}, \(G_\pi(x,y)\to+\infty\) uniformly as
\(\dist(x,y)\to0\).  Since it is continuous away from the diagonal, it is
also bounded below on \(\T^d\times\T^d\setminus\diag\), where
\(\diag:=\{(x,x):x\in\T^d\}\); fix
\(C_0\) such that \(G_\pi\ge-C_0\).  For \(L\ge1\), define the capped kernel
\[
 G_\pi^{(L)}(x,y)
 :=
 \begin{cases}
  \min\{G_\pi(x,y),L\},&x\ne y,\\
  L,&x=y,
 \end{cases}
\]
and
\begin{equation}\label{eq:F-cap}
 \F_L^\pi(\mu)
 :=\frac12\iint G_\pi^{(L)}(x,y)\,\dd\mu(x)\,\dd\mu(y).
\end{equation}
The value on the diagonal makes \(G_\pi^{(L)}\) bounded and continuous on
the full product space.  Thus \(\F_L^\pi\) is continuous under weak
convergence.  The family increases with \(L\), and
\(\F_L^\pi\ge-C_0/2\).

\begin{lemma}[Capped-diagonal characterization]
\label{lem:cutoff-characterization}
The lower-semicontinuous Gram energy satisfies
\begin{equation}\label{eq:F-cutoff-limit}
 \F^\pi(\mu)=\lim_{L\to\infty}\F_L^\pi(\mu)
 \quad\text{for every }\mu\in\Pp(\T^d).
\end{equation}
In particular, an atom has infinite energy, and \(\F^\pi\) is lower
semicontinuous.
\end{lemma}

\begin{proof}
Set \(I(\mu):=\sup_L\F_L^\pi(\mu)\).  For a smooth density,
\(G_\pi\) is locally integrable because \(\sigma<d\), and monotone
convergence after adding \(C_0\), together with \eqref{eq:centering}, gives
\(I(\mu)=\F^\pi(\mu)\).  Moreover,
\(\F_L^\pi\le\F^\pi\) on smooth measures.  Since \(\F_L^\pi\) is continuous,
the definition of \(\F^\pi\) as the lower-semicontinuous relaxation of the
smooth Gram form yields
\begin{equation}\label{eq:cap-below-gram}
 I(\mu)\le\F^\pi(\mu)
 \qquad\text{for every }\mu\in\Pp(\T^d).
\end{equation}

It remains to prove the reverse inequality when \(I(\mu)<\infty\).  The local
expansion and the lower bound away from the diagonal imply
\[
 G_\pi(x,y)
 \ge c\,\dist(x,y)^{-\sigma}
   \mathbf 1_{\{\dist(x,y)<r_0\}}-C.
\]
Thus \(\mu\) has finite Riesz \(\sigma\)-energy.  Let
\(\mu_\varepsilon=p_\varepsilon*\mu\), where \(p_\varepsilon\) is the
periodic heat kernel.  The nonnegative Fourier series of
\(G_0=-\Delta\kappa_a\) shows both that the principal energies of
\(\mu_\varepsilon\) are uniformly bounded and that they converge to the
principal energy of \(\mu\).

The target-dependent remainder has strictly smaller singular order.  More
precisely, by \eqref{eq:R-bound}, for some \(\tau<\sigma\),
\[
 |R_\pi(x,y)|\le C\bigl(1+\dist(x,y)^{-\tau}\bigr),
\]
where a logarithm is dominated by any positive power.  For \(\delta<r_0\),
\[
 \dist(x,y)^{-\tau}\mathbf 1_{\{\dist(x,y)<\delta\}}
 \le
 \delta^{\sigma-\tau}\dist(x,y)^{-\sigma}.
\]
The uniform principal-energy bound therefore makes the remainder uniformly
integrable under \(\mu_\varepsilon\otimes\mu_\varepsilon\).  Away from the
diagonal it is continuous, so weak convergence and then
\(\delta\downarrow0\) give convergence of the remainder integrals.  Hence
\[
 \F^\pi(\mu_\varepsilon)
 \longrightarrow
 \frac12\iint G_\pi(x,y)\,\dd\mu(x)\,\dd\mu(y)
 =I(\mu),
\]
where the last identity is monotone convergence for \(G_\pi+C_0\).  The
definition of the relaxation now gives
\(\F^\pi(\mu)\le\liminf_{\varepsilon\downarrow0}
\F^\pi(\mu_\varepsilon)=I(\mu)\).  Together with
\eqref{eq:cap-below-gram}, this proves \eqref{eq:F-cutoff-limit} when
\(I(\mu)<\infty\).  If \(I(\mu)=\infty\), then
\eqref{eq:cap-below-gram} already forces \(\F^\pi(\mu)=\infty\).

If \(\mu\{x\}=m>0\), then
\(\F_L^\pi(\mu)\ge \frac12Lm^2-C_0/2\), which proves the assertion about
atoms.  Finally, \(\F^\pi\) is a monotone supremum of the continuous
functionals \(\F_L^\pi\), after the fixed shift \(C_0/2\), and is therefore
lower semicontinuous.
\end{proof}

\subsection{Microscopic-to-continuum lower bound}

The next proposition is the bridge from particle dissipation to continuum
dissipation.  Its lower-bound direction is especially important: it prevents
microscopic concentration near the singular diagonal from creating an
artificially small limiting energy.

\begin{proposition}[Empirical liminf]\label{prop:empirical-liminf}
Let \(Q^N\in\Pp(\mathcal D_N)\) and let
\[
 \Lambda_N:=(\bm x\mapsto\mu_{\bm x}^N)_\#Q^N
 \in\Pp(\Pp(\T^d)).
\]
If \(\Lambda_N\rightharpoonup\Lambda\), then
\begin{equation}\label{eq:empirical-liminf}
 \int_{\Pp(\T^d)}\F^\pi(\mu)\,\dd\Lambda(\mu)
 \le
 \liminf_{N\to\infty}
 \int_{\mathcal D_N}\F_N^\pi(\bm x)\,\dd Q^N(\bm x).
\end{equation}
\end{proposition}

\begin{proof}
 Fix \(L\ge1\).  Since \(G_\pi^{(L)}\le G_\pi\) off the diagonal,
 \begin{align*}
  \F_N^\pi(\bm x)
  &\ge
  \frac1{2N^2}\sum_{i\ne j}G_\pi^{(L)}(x_i,x_j)\\
  &=\F_L^\pi(\mu_{\bm x}^N)-\frac{L}{2N}.
 \end{align*}
 The last term is exactly the diagonal contribution of the capped kernel.
 Therefore
 \[
  \liminf_{N\to\infty}\int\F_N^\pi\,\dd Q^N
  \ge
  \lim_{N\to\infty}\int\F_L^\pi(\mu)\,\dd\Lambda_N(\mu)
  =
  \int\F_L^\pi(\mu)\,\dd\Lambda(\mu),
 \]
 where \(L/(2N)\to0\), and the last equality uses continuity and boundedness
 of \(\F_L^\pi\).  Let \(L\to\infty\) and apply monotone convergence after
 adding \(C_0/2\).
\end{proof}

\begin{proposition}[Qualitative correction]\label{prop:qualitative-correction}
For every \(0<\sigma<d\), the numbers \(m_N,c_N\) defined in
\eqref{eq:qualitative-correction} satisfy
\[
 -\infty<m_N\le0,
 \qquad
 c_N\longrightarrow0.
\]
Consequently, \(\mathcal E_N^\pi=\F_N^\pi+c_N\ge0\).
\end{proposition}

\begin{proof}
The global off-diagonal lower bound \(G_\pi\ge-C_0\) gives
\(m_N\ge-C_0/2\).  On the other hand, if \(X_1,\ldots,X_N\) are independent
with law \(\pi\), local integrability and Stein centering give
\[
 \E\F_N^\pi(X_1,\ldots,X_N)
 =\frac{N(N-1)}{2N^2}
   \iint G_\pi(x,y)\,\dd\pi(x)\,\dd\pi(y)
 =0.
\]
Hence \(m_N\le0\).

If \(c_N\not\to0\), then along a subsequence \(c_N\ge2\varepsilon\) for some
\(\varepsilon>0\).  Since \(m_N=-c_N\), choose approximate minimizers
\(\bm x^N\in\mathcal D_N\) such that
\[
 \F_N^\pi(\bm x^N)\le m_N+\varepsilon\le-\varepsilon.
\]
After extracting once more,
\(\mu_{\bm x^N}^N\rightharpoonup\mu\).  Apply
\cref{prop:empirical-liminf} with \(Q^N=\delta_{\bm x^N}\):
\[
 0\le\F^\pi(\mu)
 \le\liminf_{N\to\infty}\F_N^\pi(\bm x^N)
 \le-\varepsilon,
\]
a contradiction.  Thus \(c_N\to0\), and nonnegativity follows from its
definition.
\end{proof}

\begin{remark}[Why the comparison is uniform in the horizon]
The lower-bound direction replaces \(G_\pi\) by a pointwise smaller capped
kernel and pays only its explicit diagonal term.  The lower-order,
possibly sign-indefinite part of \(G_\pi\) is never separated from the full
kernel.  No time horizon enters \eqref{eq:empirical-liminf}; this is precisely
what permits the simultaneous limit in \cref{thm:ordered}.
\end{remark}

\subsection{Time-averaged empirical laws}

For \(T>0\), define
\begin{equation}\label{eq:Lambda-NT}
 \Lambda_{N,T}
 :=\frac1T\int_0^T
 (\bm x\mapsto\mu_{\bm x}^N)_\#P_t^N\,\dd t.
\end{equation}
Its barycenter is exactly the measure in \eqref{eq:tagged-average-intro}:
\begin{equation}\label{eq:barycenter}
 \int_{\Pp(\T^d)}\mu\,\dd\Lambda_{N,T}(\mu)
 =\overline P_{\mathrm{av},T}^{N}.
\end{equation}

\section{Target identification and long-time sampling}
\label{sec:identification}

The preceding section passes the off-diagonal particle energy to every weak
limit of empirical-measure laws.  It remains to identify the zero set of the
continuum energy.  The Riesz Fourier multiplier identifies \(\F^\pi\) with a
negative Sobolev norm of \(\mu-\pi\), so \(\pi\) is its unique zero.

\subsection{The zero set of the Riesz KSD}

The following coercive identification is the point at which the Riesz Fourier
structure matters globally.

\begin{proposition}[Negative-Sobolev equivalence]\label{prop:identification}
Under \cref{ass:target}, there are constants \(0<c\le C<\infty\) such that
for every \(\mu\in\Pp(\T^d)\),
\begin{equation}\label{eq:sobolev-equivalence}
 c\|\mu-\pi\|_{\dot H^{1-a}}^2
 \le 2\F^\pi(\mu)
 \le C\|\mu-\pi\|_{\dot H^{1-a}}^2,
\end{equation}
with the convention that both sides may be infinite.  Consequently
\begin{equation}\label{eq:unique-zero}
 \F^\pi(\mu)=0\quad\Longleftrightarrow\quad\mu=\pi.
\end{equation}
\end{proposition}

\begin{proof}
Let \(\Pi f:=f-\int_{\T^d}f\) denote removal of the zero Fourier mode,
acting componentwise on vector distributions, and, for a zero-mass
distribution \(\nu\), define
\[
 P_\pi\nu:=\Pi(\nabla\nu-\bpi\nu).
\]
For smooth zero-mass \(\nu\), the Fourier form of \eqref{eq:gram} gives
\begin{equation}\label{eq:Ppi-energy}
 2\F^\pi(\pi+\nu)=\|P_\pi\nu\|_{\dot H^{-a}}^2,
\end{equation}
where the left-hand side denotes the natural quadratic extension to signed
perturbations.  The upper estimate in
\eqref{eq:sobolev-equivalence} follows because differentiation maps
\(\dot H^{1-a}\) to \(\dot H^{-a}\), while multiplication by \(\bpi\) is
bounded on \(H^{1-a}\) under \cref{ass:target}.

For the lower estimate, it suffices to prove that
\[
 \|\nu\|_{\dot H^{1-a}}
 \le C\|P_\pi\nu\|_{\dot H^{-a}}
 \qquad\left(\int_{\T^d}\nu=0\right).
\]
Otherwise there are smooth zero-mass \(\nu_n\) with
\(\|\nu_n\|_{\dot H^{1-a}}=1\) and
\(P_\pi\nu_n\to0\) in \(\dot H^{-a}\).  After passing to a subsequence,
\(\nu_n\rightharpoonup\nu\) in \(H^{1-a}\).  Multiplication by \(\bpi\),
viewed from \(H^{1-a}\) to \(H^{-a}\), is compact: it is bounded at the
first regularity and is followed by the compact one-order Sobolev embedding.
Hence
\[
 \nabla\nu_n
 =P_\pi\nu_n+\bpi\nu_n-\int_{\T^d}\bpi\nu_n
\]
converges in \(H^{-a}\).  Indeed,
\(\bpi\nu_n\to\bpi\nu\) strongly there, its zero Fourier mode also
converges, and distributional convergence identifies the resulting strong
limit with \(\nabla\nu\).  The zero-mass Poincar\'e equivalence
\(\|\nu_n-\nu\|_{\dot H^{1-a}}
\asymp\|\nabla(\nu_n-\nu)\|_{\dot H^{-a}}\) then gives strong convergence in
\(\dot H^{1-a}\), so \(\|\nu\|_{\dot H^{1-a}}=1\) and \(P_\pi\nu=0\).

The last identity means that
\[
 \nabla\nu-\bpi\nu=c
\]
for a constant vector \(c\).  Since \(\bpi=\nabla\log\pi\),
\[
 \pi\nabla\left(\frac{\nu}{\pi}\right)=c.
\]
Integrating each coordinate derivative on the torus gives
\[
 0=c_j\int_{\T^d}\pi^{-1}\,\dd x,
\]
and therefore \(c=0\).  Thus \(\nu/\pi\) is constant; the zero-mass condition
forces \(\nu=0\), a contradiction.  This proves the lower estimate.

It remains to pass from smooth densities to the relaxed energy.  Set
\(\nu=\mu-\pi\).  If \(\nu\in H^{1-a}\), then
\(\mu_\varepsilon=p_\varepsilon*\mu\) is a smooth probability measure and
\(\mu_\varepsilon-\pi\to\nu\) strongly in \(H^{1-a}\).  Boundedness of
\(P_\pi\) and \eqref{eq:Ppi-energy} provide a recovery sequence whose Gram
energies converge to \(\frac12\|P_\pi\nu\|_{\dot H^{-a}}^2\).  Conversely,
if smooth probabilities \(\mu_n\rightharpoonup\mu\) have bounded Gram
energy, the lower estimate just proved bounds \(\mu_n-\pi\) in
\(H^{1-a}\).  Every weak Sobolev limit agrees distributionally with
\(\mu-\pi\), and weak lower semicontinuity gives the matching lower bound.
Thus the relaxed energy equals
\(\frac12\|P_\pi(\mu-\pi)\|_{\dot H^{-a}}^2\) when
\(\mu-\pi\in H^{1-a}\), and is \(+\infty\) otherwise.  The same estimate
appears in
\citet[Lemma 4.2]{ChizatColomboColomboFernandezReal2026}.  Finally, the lower
bound shows that \(\F^\pi(\mu)=0\) only when \(\mu-\pi=0\), proving
\eqref{eq:unique-zero}.
\end{proof}

\subsection{Proofs of the long-time conclusions}

\begin{proof}[Proof of \cref{thm:ordered}]
Consider an arbitrary subsequence in \(N\).  Compactness of
\(\Pp(\Pp(\T^d))\) gives a further subsequence such that
\(\Lambda_{N,T_N}\rightharpoonup\Lambda\).  Apply
\cref{prop:empirical-liminf} to
\[
 Q^{N,T_N}:=\frac1{T_N}\int_0^{T_N}P_t^N\,\dd t.
\]
The entropy identity gives
\[
 A_N:=\int\F_N^\pi\,\dd Q^{N,T_N}
 =\frac{H_N(0)-H_N(T_N)}{2T_N}
 \le\frac{H_N(0)}{2T_N}.
\]
Consequently,
\[
 0\le\int\F^\pi(\mu)\,\dd\Lambda(\mu)
 \le\liminf_NA_N
 \le\limsup_NA_N
 \le0.
\]
By \cref{prop:identification}, \(\F^\pi\) is nonnegative and vanishes only at
\(\pi\), so \(\Lambda=\delta_\pi\).  Every subsequential limit is the same;
hence \(\Lambda_{N,T_N}\rightharpoonup\delta_\pi\).

The barycenter map on \(\Pp(\Pp(\T^d))\) is weakly continuous, and
\eqref{eq:barycenter} identifies the barycenter of \(\Lambda_{N,T_N}\) with
\(\overline P_{\mathrm{av},T_N}^{N}\).  This proves
\eqref{eq:ordered-main}.
\end{proof}

\begin{remark}[Fixed-horizon consequence]
The same compactness argument also gives a fixed-horizon estimate.  If
\(\sup_NH_N(0)\le H_*\), define
\[
 \omega_\pi(r):=
 \sup\left\{\dBL(\mu,\pi):
 \mu\in\Pp(\T^d),\ \F^\pi(\mu)\le r\right\}.
\]
Lower semicontinuity, compactness, and \eqref{eq:unique-zero} imply
\(\omega_\pi(r)\downarrow0\) as \(r\downarrow0\).  Applying
\cref{prop:empirical-liminf} at fixed \(T\), followed by convexity of
\(\F^\pi\) under barycenters, gives
\[
 \limsup_{N\to\infty}
 \dBL\!\left(\overline P_{\mathrm{av},T}^{N},\pi\right)
 \le\omega_\pi\!\left(\frac{H_*}{2T}\right).
\]
\end{remark}

\begin{proof}[Proof of \cref{cor:stationary}]
Since \(Q^N\ll\pi^{\otimes N}\) and the collision set is
\(\pi^{\otimes N}\)-null, \(Q^N\) may be regarded as a law on
\(\mathcal D_N\).  Stationarity and \eqref{eq:entropy-identity} imply
\[
 \int_{\mathcal D_N}\F_N^\pi\,\dd Q^N=0.
\]
Let \(\Lambda_N\) be the law of the empirical measure under \(Q^N\).
Along any convergent subsequence, \cref{prop:empirical-liminf} yields
\[
 \int\F^\pi(\mu)\,\dd\Lambda(\mu)
 \le\liminf_{N\to\infty}\int\F_N^\pi\,\dd Q^N=0.
\]
Nonnegativity and \eqref{eq:unique-zero} force
\(\Lambda=\delta_\pi\).  Thus every subsequential limit is \(\delta_\pi\).
The barycenter of \(\Lambda_N\) is
\(N^{-1}\sum_iQ_i^N\), so this label-averaged marginal converges to \(\pi\).
If \(Q^N\) is exchangeable, the barycenter equals \(Q_1^N\).
\end{proof}

\section{Quantitative renormalization below the logarithmic threshold}
\label{sec:quantitative-onsager}

The qualitative correction \(c_N\to0\) follows from compactness and applies
throughout \(0<\sigma<d\).  For \(0<\sigma<2\), the target-dependent part
of the Stein kernel can also be separated from a positive
translation-invariant minorant.  This gives the explicit rate in
\cref{thm:quantitative-onsager}.  Throughout this section, assume
\[
 0<\sigma<2,
 \qquad
 s:=d-\sigma=2a-2\in(0,d).
\]
The operator factorization and Bessel minorant below apply throughout this
range.

\subsection{Normalized Stein operator}

Let \(\mathsf L=(-\Delta)^{1/2}\) on mean-zero periodic distributions, with
multiplier
\(\lambda_\ell:=2\pi|\ell|\) for \(\ell\ne0\), and let \(\Pi\) denote the
mean-zero projection, componentwise for vector fields.  Set
\[
 D_{\bpi}:=\nabla-\bpi,
 \qquad
 T_\pi:=D_{\bpi}^*\mathsf L^{-2a}D_{\bpi}.
\]
The inverse power is the mean-zero pseudoinverse.  The Gram calculation in
\cref{prop:gram} can equivalently be written, for smooth signed \(f\), as
\begin{equation}\label{eq:operator-form}
 \iint G_\pi(x,y)f(x)f(y)\,\dd x\,\dd y
 =\left\langle D_{\bpi}f,\mathsf L^{-2a}D_{\bpi}f\right\rangle.
\end{equation}
For \(\bpi=0\), the corresponding operator is
\[
 T_0=\mathsf L^{2-2a}=\mathsf L^{-s}.
\]

We normalize the principal part to the identity.  On the mean-zero subspace
\(L_0^2(\T^d)\), define
\[
 \mathcal R:=\nabla\mathsf L^{-1},
 \qquad
 \mathcal C:=\mathsf L^{-a}\Pi M_{\bpi}\mathsf L^{a-1}.
\]
Here \(M_{\bpi}\) denotes multiplication by the vector field \(\bpi\).
Since \(\mathcal R^*\mathcal R=I\) on \(L_0^2\),
\begin{equation}\label{eq:normalized-B}
 \mathcal B
 :=\mathsf L^{s/2}T_\pi\mathsf L^{s/2}
 =(\mathcal R-\mathcal C)^*(\mathcal R-\mathcal C)
 =I+\mathcal K.
\end{equation}
The next lemma shows that the normalized target perturbation gains two
derivatives as a quadratic form.

\begin{lemma}[Two-order target perturbation]
\label{lem:two-order-perturbation}
There is \(C_\pi<\infty\) such that
\begin{equation}\label{eq:two-order-form}
 |\langle u,\mathcal K v\rangle|
 \le C_\pi\|u\|_{H^{-1}}\|v\|_{H^{-1}}
 \qquad(u,v\in L_0^2(\T^d)).
\end{equation}
\end{lemma}

\begin{proof}
Put \(\alpha=m-1\).  The parameter assumptions give
\[
 \alpha>\frac d2+1,
 \qquad
 \alpha>a,
\]
the second inequality following from
\(a<1+d/2\).  We first record the periodic commutator estimate
\begin{equation}\label{eq:commutator-estimate}
 \|[\mathsf L^a,b]g\|_{H^{1-a}}
 \le C\|b\|_{H^\alpha}\|g\|_2
\end{equation}
for each component \(b\) of \(\bpi\).  By skew-adjointness of the
commutator, it is enough to prove
\[
 \|[\mathsf L^a,b]h\|_2
 \le C\|b\|_{H^\alpha}\|h\|_{H^{a-1}}.
\]
The periodic Kato--Ponce inequality, obtained from the Euclidean estimate by
standard transference, gives
\[
 \|[\mathsf L^a,b]h\|_2
 \lesssim
 \|\nabla b\|_\infty\|h\|_{H^{a-1}}
 +\|\mathsf L^ab\|_{L^p}\|h\|_{L^q},
 \qquad
 \frac1p+\frac1q=\frac12;
\]
see, for example, \citet{KatoPonce1988}.  Write
\(\delta=\alpha-a>0\).  If \(0<\delta<d/2\), take
\[
 \frac1p=\frac12-\frac{\delta}{d},
 \qquad
 \frac1q=\frac{\delta}{d}.
\]
If \(\delta>d/2\), take \(p=\infty\), \(q=2\).  At the borderline
\(\delta=d/2\), choose \(0<\varepsilon<a-1\), use
\(H^{d/2}\hookrightarrow H^{d/2-\varepsilon}\hookrightarrow L^p\) with
\(1/p=\varepsilon/d\), and take
\(1/q=1/2-\varepsilon/d\).  In every case,
\(\alpha-1>d/2\) supplies the required embedding of
\(H^{a-1}\) into \(L^q\), while
\(H^\alpha\hookrightarrow W^{1,\infty}\).  This proves
\eqref{eq:commutator-estimate}.

The mean-zero projection is important in the following exact decomposition:
\begin{equation}\label{eq:C-decomposition}
 \mathcal C
 =\Pi M_{\bpi}\mathsf L^{-1}+\mathcal E,
 \qquad
 \mathcal E
 :=-\mathsf L^{-a}\Pi[\mathsf L^a,M_{\bpi}]\mathsf L^{-1}.
\end{equation}
Indeed, this is just
\(M_{\bpi}\mathsf L^a
=\mathsf L^aM_{\bpi}-[\mathsf L^a,M_{\bpi}]\), with the homogeneous
powers interpreted as mean-zero pseudoinverses.  Estimate
\eqref{eq:commutator-estimate} gives
\[
 \mathcal E:H^{-1}\longrightarrow H^1.
\]
Sobolev multiplication by \(\bpi\in H^\alpha\) also gives
\[
 \mathcal C:H^{-1}\longrightarrow L^2.
\]

Expanding \eqref{eq:normalized-B},
\[
 \mathcal K
 =-\mathcal R^*\mathcal C-\mathcal C^*\mathcal R
  +\mathcal C^*\mathcal C.
\]
The nominal order-\(-1\) contributions in
\eqref{eq:C-decomposition} cancel exactly:
\begin{equation}\label{eq:order-one-cancellation}
 -\mathcal R^*\Pi M_{\bpi}\mathsf L^{-1}
 -(\Pi M_{\bpi}\mathsf L^{-1})^*\mathcal R
 =
 \mathsf L^{-1}M_{\operatorname{div}\bpi}\mathsf L^{-1}.
\end{equation}
The projections may be dropped in this calculation because the range of
\(\mathcal R\) is mean zero and \(\mathcal R^*\) annihilates constants.
Identity \eqref{eq:order-one-cancellation} follows componentwise from
\[
 \partial_jM_{(\bpi)_j}-M_{(\bpi)_j}\partial_j
 =M_{\partial_j(\bpi)_j}.
\]
Since \(\operatorname{div}\bpi\in L^\infty\), the right-hand side maps
\(H^{-1}\) to \(H^1\).  Every remaining term contains \(\mathcal E\) or
\(\mathcal C^*\mathcal C\), and the two mapping properties above give the
same \(H^{-1}\to H^1\) bound.  Duality yields
\eqref{eq:two-order-form}.
\end{proof}

\begin{lemma}[Strict positivity of the normalized operator]
\label{lem:B-positive}
There is \(\beta>0\), depending only on \(d,a,\pi\), such that
\begin{equation}\label{eq:B-positive}
 \mathcal B\ge\beta I
 \qquad\text{on }L_0^2(\T^d).
\end{equation}
\end{lemma}

\begin{proof}
By \eqref{eq:normalized-B}, \(\mathcal B\ge0\).  The preceding lemma shows
that \(\mathcal K\) maps \(H^{-1}\) to \(H^1\), hence is compact on
\(L_0^2\).  It remains to exclude a kernel.

Suppose \(\mathcal Bu=0\) and set \(f=\mathsf L^{a-1}u\).  Then
\(\mathsf L^{-a}\Pi D_{\bpi}f=0\), so
\[
 D_{\bpi}f=c
\]
for a constant vector \(c\).  Since
\[
 D_{\bpi}f
 =\nabla f-\bpi f
 =\pi\nabla\left(\frac f\pi\right),
\]
we have \(\nabla(f/\pi)=c/\pi\).  Integrating each coordinate derivative on
the torus gives
\[
 0=c_j\int_{\T^d}\pi^{-1}\,\dd x,
\]
and therefore \(c=0\).  Thus \(f=c_0\pi\).  But \(f\) has zero Lebesgue mean,
whereas \(\int\pi=1\), so \(c_0=0\) and then \(u=0\).  Hence
\(\mathcal B=I+\mathcal K\) is injective.  Its spectrum can accumulate only
at \(1\), and \eqref{eq:B-positive} follows.
\end{proof}

\subsection{A positive Bessel minorant}

For \(M\ge1\), let \(J_M\) be the periodic Bessel kernel with Fourier
coefficients
\begin{equation}\label{eq:Bessel-multiplier}
 \widehat J_M(\ell)
 =\bigl(M^2+\lambda_\ell^2\bigr)^{-s/2},
 \qquad \ell\in\mathbb Z^d,
\end{equation}
where \(\lambda_0=0\).  Its heat representation is
\begin{equation}\label{eq:Bessel-heat}
 J_M(z)
 =\frac1{\Gamma(s/2)}
 \int_0^\infty t^{s/2-1}e^{-M^2t}p_t(z)\,\dd t.
\end{equation}
The periodic heat kernel is nonnegative, so \(J_M\ge0\).  Moreover,
\[
 J_M\in L^1(\T^d),
 \qquad
 \int_{\T^d}J_M\,\dd x=M^{-s}.
\]

Center this kernel with respect to \(\pi\):
\begin{equation}\label{eq:centered-Bessel}
 \overline J_M^\pi(x,y)
 :=J_M(x-y)-u_M(x)-u_M(y)+\gamma_M,
\end{equation}
where
\[
 u_M:=J_M*\pi,
 \qquad
 \gamma_M:=\int_{\T^d}u_M\,\dd\pi.
\]
Then \(\overline J_M^\pi\) is centered in each variable,
\[
 0\le u_M(x)\le\|\pi\|_\infty M^{-s},
 \qquad
 \gamma_M\ge0,
\]
and \(u_M\) is continuous because \(J_M\in L^1\) and \(\pi\) is continuous.

\begin{proposition}[Bessel minorant]\label{prop:Bessel-minorant}
There is \(M_0<\infty\) such that, for every \(M\ge M_0\),
\begin{equation}\label{eq:Bessel-minorant}
 T_\pi-J_M*\ge0
\end{equation}
as a quadratic form on smooth zero-mass functions.
\end{proposition}

\begin{proof}
Conjugating by \(\mathsf L^{s/2}\), the Bessel convolution operator becomes
the multiplier \(\mathcal J_M\) with
\[
 \widehat{\mathcal J_M}(\ell)
 =\chi_M(\ell)
 :=\left(\frac{\lambda_\ell^2}
 {M^2+\lambda_\ell^2}\right)^{s/2},
 \qquad \ell\ne0.
\]
Thus
\[
 \mathsf L^{s/2}(T_\pi-J_M*)\mathsf L^{s/2}
 =\mathcal B-\mathcal J_M
 =\mathcal D_M+\mathcal K,
 \qquad
 \mathcal D_M:=I-\mathcal J_M.
\]
The scalar inequality
\[
 1-(1+x)^{-s/2}\ge c_s\frac{x}{1+x},
 \qquad x\ge0,
\]
gives
\begin{equation}\label{eq:D-lower-symbol}
 1-\chi_M(\ell)
 \ge c_s\frac{M^2}{M^2+\lambda_\ell^2}.
\end{equation}

Let \(P_R\) project onto \(0<\lambda_\ell\le R\), let \(Q_R=I-P_R\), and
write \(u=u_{\rm lo}+u_{\rm hi}\), where \(u_{\rm lo}=P_Ru\) and
\(u_{\rm hi}=Q_Ru\).  If \(M\ge R\ge1\),
\eqref{eq:D-lower-symbol} implies
\begin{equation}\label{eq:D-high}
 \langle u_{\rm hi},\mathcal D_Mu_{\rm hi}\rangle
 \ge c_s'R^2\|u_{\rm hi}\|_{H^{-1}}^2.
\end{equation}
Indeed,
\[
 \frac{M^2(1+\lambda^2)}{M^2+\lambda^2}
 \ge\frac{1+R^2}{2}
 \qquad(\lambda\ge R,\ M\ge R).
\]

Choose \(R\) so large that
\[
 \frac{C_\pi}{c_s'R^2}\le\frac14,
 \qquad
 \frac{4C_\pi^2}{\beta c_s'R^2}\le\frac14.
\]
Then \eqref{eq:two-order-form} and \eqref{eq:D-high} give
\begin{equation}\label{eq:Bessel-high}
 \langle u_{\rm hi},(\mathcal D_M+\mathcal K)u_{\rm hi}\rangle
 \ge\frac34\langle u_{\rm hi},\mathcal D_Mu_{\rm hi}\rangle.
\end{equation}
On the finite-dimensional low block,
\(\mathcal B\ge\beta I\), while
\[
 \|\mathcal J_MP_R\|_{2\to2}
 \le\left(\frac{R^2}{M^2+R^2}\right)^{s/2}\longrightarrow0.
\]
After increasing \(M_0=M_0(R)\),
\begin{equation}\label{eq:Bessel-low}
 \langle u_{\rm lo},(\mathcal B-\mathcal J_M)u_{\rm lo}\rangle
 \ge\frac{3\beta}{4}\|u_{\rm lo}\|_2^2.
\end{equation}
Because \(\mathcal D_M\) is diagonal in Fourier variables, the low--high
cross term comes only from \(\mathcal K\).  A final use of
\eqref{eq:two-order-form}, \eqref{eq:D-high}, and Young's inequality gives
\begin{equation}\label{eq:Bessel-cross}
 2|\langle u_{\rm lo},\mathcal Ku_{\rm hi}\rangle|
 \le\frac{\beta}{4}\|u_{\rm lo}\|_2^2
 +\frac14\langle u_{\rm hi},\mathcal D_Mu_{\rm hi}\rangle.
\end{equation}
Combining \eqref{eq:Bessel-high}--\eqref{eq:Bessel-cross},
\[
 \langle u,(\mathcal B-\mathcal J_M)u\rangle
 \ge\frac{\beta}{2}\|u_{\rm lo}\|_2^2
 +\frac12\langle u_{\rm hi},\mathcal D_Mu_{\rm hi}\rangle
 \ge0.
\]
Conjugating back proves \eqref{eq:Bessel-minorant}.
\end{proof}

\subsection{A continuous positive-semidefinite remainder}

In this subsection and the next, assume \(0<\sigma<2\).

For \(x\ne y\), define
\begin{equation}\label{eq:KM-definition}
 K_M^\pi(x,y):=G_\pi(x,y)-\overline J_M^\pi(x,y).
\end{equation}
Both terms are singular on the diagonal; the next lemma shows that their
difference has a continuous extension.

\begin{lemma}[Continuous remainder]\label{lem:KM-continuous}
For every \(M\ge M_0\), the kernel \(K_M^\pi\) extends continuously to
\(\T^d\times\T^d\), is positive semidefinite, and satisfies
\begin{equation}\label{eq:KM-diagonal}
 \sup_{x\in\T^d}K_M^\pi(x,x)\le CM^\sigma.
\end{equation}
\end{lemma}

\begin{proof}
Write
\[
 G_\pi(x,y)=G_0(y-x)+R_\pi(x,y),
 \qquad
 G_0=-\Delta\kappa_a.
\]
Because \(\sigma<2\), \(\kappa_a\) is continuous at the origin.  Moreover,
\[
 |\nabla\kappa_a(z)|\lesssim1+|z|^{1-\sigma},
\]
and \(\bpi\) is Lipschitz.  Hence
\[
 |(\bpi(x)-\bpi(y))\cdot\nabla\kappa_a(y-x)|
 \lesssim |x-y|+|x-y|^{2-\sigma}\longrightarrow0.
\]
The remaining term
\(\bpi(x)\cdot\bpi(y)\kappa_a(y-x)\) is also continuous across the
diagonal.  Thus \(R_\pi\) extends continuously.

The Fourier coefficients of \(G_0-J_M\) are \(-M^{-s}\) at zero and, for
\(\ell\ne0\),
\[
 \lambda_\ell^{-s}
 -\bigl(M^2+\lambda_\ell^2\bigr)^{-s/2}.
\]
Uniformly in \(M\ge1\) and \(\ell\ne0\),
\begin{equation}\label{eq:Bessel-coefficient-bound}
 0\le
 \lambda_\ell^{-s}
 -\bigl(M^2+\lambda_\ell^2\bigr)^{-s/2}
 \le C_s\min\{\lambda_\ell^{-s},
 M^2\lambda_\ell^{-s-2}\}.
\end{equation}
Since \(s+2>d\) is equivalent to \(\sigma<2\), this sequence is absolutely
summable.  Therefore \(G_0-J_M\) is continuous.  Splitting the sum at
\(\lambda_\ell=M\) gives
\begin{equation}\label{eq:G0-J-diagonal}
 (G_0-J_M)(0)
 \le
 C\sum_{0<\lambda_\ell\le M}\lambda_\ell^{-s}
 +CM^2\sum_{\lambda_\ell>M}\lambda_\ell^{-s-2}
 \le CM^{d-s}=CM^\sigma.
\end{equation}

Using \eqref{eq:centered-Bessel}, the off-diagonal expression
\eqref{eq:KM-definition} is
\[
 (G_0-J_M)(y-x)+R_\pi(x,y)+u_M(x)+u_M(y)-\gamma_M.
\]
Every term now has a continuous extension.  The bounds on \(u_M,\gamma_M\),
\eqref{eq:G0-J-diagonal}, and boundedness of the extended \(R_\pi\) prove
\eqref{eq:KM-diagonal}.

It remains to prove positive semidefiniteness.  Both \(G_\pi\) and
\(\overline J_M^\pi\) are centered with respect to \(\pi\).  For any finite
signed measure \(\nu\), put
\(\nu_0=\nu-\nu(\T^d)\pi\).  Centering gives
\[
 \iint K_M^\pi\,\dd\nu\,\dd\nu
 =\iint K_M^\pi\,\dd\nu_0\,\dd\nu_0.
\]
For smooth \(\nu_0\) of zero mass, the right-hand side is nonnegative by
\cref{prop:Bessel-minorant}.  Periodic mollification and continuity of
\(K_M^\pi\) extend this inequality to every finite signed measure.  Hence
\(K_M^\pi\) is an ordinary continuous positive-semidefinite kernel.
\end{proof}

\subsection{Completion of the quantitative estimate}

\begin{proof}[Proof of \cref{thm:quantitative-onsager} for
\(0<\sigma<2\)]
Fix \(M\ge M_0\).  Positive semidefiniteness and
\eqref{eq:KM-diagonal} give
\[
 \sum_{i\ne j}K_M^\pi(x_i,x_j)
 \ge-\sum_iK_M^\pi(x_i,x_i)
 \ge-CNM^\sigma.
\]
Therefore
\begin{equation}\label{eq:KM-offdiagonal}
 \frac1{2N^2}\sum_{i\ne j}K_M^\pi(x_i,x_j)
 \ge-C\frac{M^\sigma}{N}.
\end{equation}

For the centered Bessel kernel,
\[
 \begin{aligned}
 \sum_{i\ne j}\overline J_M^\pi(x_i,x_j)
 ={}&
 \sum_{i\ne j}J_M(x_i-x_j)
 -2(N-1)\sum_i u_M(x_i)\\
 &+N(N-1)\gamma_M.
 \end{aligned}
\]
The first and third terms are nonnegative, while
\(u_M\le\|\pi\|_\infty M^{-s}\).  Thus
\begin{equation}\label{eq:Bessel-offdiagonal}
 \frac1{2N^2}\sum_{i\ne j}\overline J_M^\pi(x_i,x_j)
 \ge-CM^{-s}.
\end{equation}
Since \(G_\pi=K_M^\pi+\overline J_M^\pi\) off the diagonal,
\eqref{eq:KM-offdiagonal}--\eqref{eq:Bessel-offdiagonal} yield
\[
 \F_N^\pi(\bm x)
 \ge-C\left(\frac{M^\sigma}{N}+M^{-s}\right).
\]
Choose \(M=N^{1/d}\).  Since \(s+\sigma=d\),
\[
 \frac{M^\sigma}{N}=M^{-s}=N^{-1+\sigma/d}.
\]
For the finitely many \(N\) with \(N^{1/d}<M_0\), enlarge the constant.
This proves \eqref{eq:onsager}, and
\eqref{eq:quantitative-correction} follows from the definition of \(c_N\).
\end{proof}

\begin{remark}[The endpoint obstruction]
The cancellation above gains exactly two orders.  At \(\sigma=2\),
\eqref{eq:Bessel-coefficient-bound} has a high-frequency
\(|\ell|^{-d}\) tail, so the residual diagonal diverges logarithmically and
\(K_M^\pi\) is no longer a continuous finite-diagonal kernel.  For
\(\sigma>2\), the residual singularity is stronger.  A further
renormalization is therefore needed for \(2\le\sigma<d\).
\end{remark}

\section{Open directions}\label{sec:open}

Three extensions remain particularly natural.

\begin{enumerate}
\item The qualitative correction \(c_N\to0\) holds for the full range
\(0<\sigma<d\), but an explicit rate in the range
\(2\le\sigma<d\) remains open.  At and above the endpoint, the
post-cancellation remainder no longer has a finite diagonal.
The ball-growth and smearing methods of
\citet{PetracheSerfaty2017,NguyenRosenzweigSerfaty2022} control the
translation-invariant principal Riesz interaction, but the target-dependent
remainder is variable-coefficient and sign-indefinite.  An explicit rate
would require a corresponding commutator or pseudodifferential truncation
estimate.

\item A last-iterate result on logarithmic time scales would follow from a
singular propagation-of-chaos estimate with controlled time dependence.
Existing Riesz commutator estimates address the translation-invariant term
\(-\nabla\kappa_a*(\mu^N-\rho)\) in their admissible range, but not the
discrete trace generated by the target-weighted term
\(\kappa_a*(\bpi(\mu^N-\rho))\).  Controlling that trace would connect the
population decay of
\citet{ChizatColomboColomboFernandezReal2026} to a particle last-iterate
limit.

\item It would be useful to admit deterministic atomic initialization.  Such
initial laws have infinite entropy relative to \(\pi^{\otimes N}\), so the
joint-entropy argument would need an initial-layer estimate or a replacement
that measures a finite defect.
\end{enumerate}

An extension to \(\R^d\) is also natural, but it brings separate questions of
tightness, confinement, and long-range control.

\section*{Acknowledgments}
The authors used OpenAI’s ChatGPT during the preparation of this manuscript to assist with mathematical exploration, proof checking, and improvements to the language and readability. The authors reviewed and verified all resulting content and take full responsibility for the manuscript.

\appendix
\section{Periodic Riesz asymptotics}\label{app:riesz}

\begin{lemma}\label{lem:riesz-asymptotics}
Let \(1<a<1+d/2\) and \(\sigma=d+2-2a\).  Near \(z=0\), the periodic kernel
\(\kappa_a\) has the same singular part as the Euclidean Riesz potential of
order \(2a\).  In particular,
\begin{align}
 -\Delta\kappa_a(z)
 &=c_{d,a}|z|^{-\sigma}+O(1),\label{eq:asymp-laplacian}\\
 -z\cdot\nabla\kappa_a(z)
 &\ge c_{d,a}'|z|^{2-\sigma}\label{eq:asymp-radial}
\end{align}
for sufficiently small \(z\ne0\).  Moreover,
\[
 \kappa_a(z)=
 \begin{cases}
 c|z|^{2-\sigma}+O(1),&\sigma>2,\\
 c\log(1/|z|)+h(z),\quad h\in C^\infty\text{ near }0,&\sigma=2,\\
 \kappa_a(0)-c|z|^{2-\sigma}+o(|z|^{2-\sigma}),&\sigma<2,
 \end{cases}
\]
with constants whose signs are consistent with
\eqref{eq:asymp-laplacian}.  In addition,
\[
 |\nabla^2\kappa_a(z)|
 \le C\bigl(1+\dist(z,0)^{-\sigma}\bigr)
 \qquad(z\ne0).
\]
\end{lemma}

\begin{proof}
Choose a cutoff supported in a coordinate chart around the origin.  Poisson
summation, or equivalently the standard parametrix for the elliptic
pseudodifferential operator \((-\Delta)^a\), decomposes the periodic
fundamental solution into the Euclidean homogeneous fundamental solution plus
a smooth periodic remainder.  The Euclidean homogeneity is \(2a-d=2-\sigma\),
with the logarithmic replacement at zero homogeneity.  Applying \(-\Delta\)
lowers the degree by two and gives degree \(-\sigma\).  Radial
differentiation gives \eqref{eq:asymp-radial}; the coefficient is positive
because \((-\Delta)^a\kappa_a=\delta_0-1\) and
\(\widehat\kappa_a(\ell)>0\).  Differentiating the homogeneous singular
part twice, and using smoothness of the periodic remainder, gives the Hessian
bound.
\end{proof}

\subsection{The pair-energy gradient near collisions}

The next estimate rules out cancellation of the singular pair forces inside a
multiscale cluster.  It is the input used in
\cref{lem:pair-energy-barrier}.

\begin{lemma}[Closest-scale virial estimate]
\label{lem:pair-energy-gradient}
Fix \(N\ge2\) and \(2\le\sigma<d\).  There are \(c,r_0>0\) such that
\[
 |\nabla\mathcal W_N(\bm x)|
 \ge c\,r(\bm x)^{1-\sigma},
 \qquad
 r(\bm x):=\min_{i\ne j}\dist(x_i,x_j),
\]
whenever \(0<r(\bm x)<r_0\).
\end{lemma}

\begin{proof}
Argue by contradiction.  Suppose that for a sequence
\(\bm x^n\in\mathcal D_N\), with \(r_n:=r(\bm x^n)\to0\),
\[
 r_n^{\sigma-1}|\nabla\mathcal W_N(\bm x^n)|\longrightarrow0.
\]
After passing to a subsequence, the pair realizing \(r_n\) is fixed and, for
every pair \(i,j\), the ratio
\(\dist(x_i^n,x_j^n)/r_n\) converges in \([1,\infty]\).  Declare
\(i\sim j\) when this limit is finite.  The triangle inequality makes this
an equivalence relation.  Let \(I\) be the class containing the closest
pair.  Then \(|I|\ge2\), every internal distance is comparable to \(r_n\),
and
\begin{equation}\label{eq:terminal-cluster-separation}
 \frac{\dist(x_i^n,x_k^n)}{r_n}\longrightarrow\infty
 \qquad(i\in I,\ k\notin I).
\end{equation}

Choose local lifts of the points in \(I\), and write
\[
 \bar x_I^n=\frac1{|I|}\sum_{i\in I}x_i^n,
 \qquad
 q_i^n=x_i^n-\bar x_I^n,
 \qquad
 R_n^2=\sum_{i\in I}|q_i^n|^2.
\]
The internal distance comparison gives \(R_n\asymp r_n\).  Pairwise
symmetrization yields
\begin{align}
 \sum_{i\in I}q_i^n\cdot\nabla_{x_i}\mathcal W_N(\bm x^n)
 ={}&
 \sum_{\substack{i<j\\i,j\in I}}
 (x_j^n-x_i^n)\cdot
 \nabla\kappa_a(x_j^n-x_i^n)
 +E_n,                                      \label{eq:pair-virial}
\end{align}
where \(E_n\) contains the interactions with particles outside \(I\).
By \cref{lem:riesz-asymptotics}, the internal sum is bounded above by
\[
 -c\Phi_\sigma(r_n),
 \qquad
 \Phi_\sigma(r):=
 \begin{cases}
  1,&\sigma=2,\\
  r^{2-\sigma},&\sigma>2.
 \end{cases}
\]

It remains to check that the exterior terms cannot cancel this contribution.
The same local asymptotics give
\[
 |\nabla^2\kappa_a(z)|
 \le C\bigl(1+\dist(z,0)^{-\sigma}\bigr).
\]
For \(k\notin I\), let
\(D_{k,n}:=\min_{i\in I}\dist(x_i^n,x_k^n)\).  Subtracting the value of the
exterior force at \(\bar x_I^n\), whose contribution vanishes because
\(\sum_iq_i^n=0\), gives
\[
 |E_n|
 \le C R_n^2\sum_{k\notin I}\bigl(1+D_{k,n}^{-\sigma}\bigr).
\]
Here \(D_{k,n}/R_n\to\infty\), so for large \(n\) the interpolation segment
from each \(x_i^n\) to \(\bar x_I^n\) remains at least
\(D_{k,n}/2\) from \(x_k^n\), which justifies the displayed Hessian bound
throughout the mean-value estimate.
By \eqref{eq:terminal-cluster-separation}, this is
\(o(\Phi_\sigma(r_n))\): for each exterior particle, the singular part
divided by \(\Phi_\sigma(r_n)\) is \(O((r_n/D_{k,n})^\sigma)\), while the
bounded part is \(O(r_n^\sigma)\).

Consequently, the absolute value of the left-hand side of
\eqref{eq:pair-virial} is at least \(c\Phi_\sigma(r_n)\) for large \(n\).
Cauchy--Schwarz and \(R_n\asymp r_n\) now give
\[
 |\nabla\mathcal W_N(\bm x^n)|
 \ge \frac{c\Phi_\sigma(r_n)}{R_n}
 \ge c r_n^{1-\sigma},
\]
contradicting the choice of the sequence.
\end{proof}

\section{Population entropy dissipation for the singular kernel}
\label{app:population-dissipation}

We justify \eqref{eq:population-dissipation-intro} without using pointwise
reproduction or evaluating the singular kernel on the diagonal.  Let
\((\rho_t)_{t\in[0,T]}\) be a smooth, strictly positive solution of
\begin{equation}\label{eq:population-flow}
 \partial_t\rho_t+\nabla\cdot(\rho_t v_t)=0,
 \qquad
 v_t(x):=
 \int_{\T^d}\mathcal A_{\pi,y}k(y,x)\rho_t(y)\,\dd y.
\end{equation}
Set
\[
 q_t:=\bpi\rho_t-\nabla\rho_t
     =-\rho_t\nabla\log\frac{\rho_t}{\pi}.
\]
Because \(\rho_t\) is smooth, the convolution of the distribution
\(\kappa_a\) with \(q_t\) is smooth.  Distributional integration by parts
in the source variable gives
\[
 v_t
 =\int_{\T^d}\mathcal A_{\pi,y}k(y,\cdot)\rho_t(y)\,\dd y
 =\kappa_a*q_t.
\]
Differentiating the relative entropy along \eqref{eq:population-flow} and
integrating by parts on the torus therefore yields
\[
\begin{aligned}
 \frac{\dd}{\dd t}\Ent(\rho_t\mid\pi)
 &=\int_{\T^d}\log\frac{\rho_t}{\pi}\,\partial_t\rho_t\,\dd x\\
 &=\int_{\T^d}\rho_t\nabla\log\frac{\rho_t}{\pi}\cdot v_t\,\dd x\\
 &=-\int_{\T^d}q_t\cdot(\kappa_a*q_t)\,\dd x.
\end{aligned}
\]
The Fourier identity established in the proof of \cref{prop:gram} gives
\[
 \int_{\T^d}q_t\cdot(\kappa_a*q_t)\,\dd x
 =\left\|
  \int_{\T^d}\mathcal A_{\pi,y}k(y,\cdot)\,\dd\rho_t(y)
 \right\|_{\mathcal H_{\kappa_a}^d}^{2}
 =\operatorname{KSD}_{\kappa_a}^2(\rho_t\mid\pi)
 =2\F^\pi(\rho_t).
\]
This proves \eqref{eq:population-dissipation-intro}.

\section{Entropy approximation}\label{app:entropy-approx}

We give the localization argument behind \cref{prop:entropy-identity}.  Fix
\(N\) and a terminal time \(t>0\), let \(\Phi_s^N\) denote the particle flow,
and write
\[
 h_0:=\frac{\dd P_0^N}{\dd\pi^{\otimes N}}.
\]
For \(m\ge1\), set
\[
 A_m:=\left\{\bm x:
  m^{-1}\le h_0(\bm x)\le m,
  \quad
  \min_{0\le s\le t}\min_{i\ne j}
  \dist\bigl(\Phi_s^N(\bm x)_i,\Phi_s^N(\bm x)_j\bigr)
  \ge m^{-1}
 \right\}.
\]
The sets \(A_m\) increase.  Moreover,
\(P_0^N(\bigcup_mA_m)=1\): the density ratio is finite and positive
\(P_0^N\)-almost surely, and \cref{prop:no-collision} implies that the minimum
pair distance along each trajectory is positive on the compact interval
\([0,t]\).  Put \(\alpha_m:=P_0^N(A_m)\) and, after discarding finitely many
empty sets, define
\[
 P_0^{N,m}:=\frac{P_0^N(\,\cdot\cap A_m)}{\alpha_m},
 \qquad
 P_s^{N,m}:=(\Phi_s^N)_\#P_0^{N,m}.
\]

On the trajectory tube generated by \(A_m\), the vector field and its first
derivatives are bounded.  The flow is a \(C^1\) one-to-one map there, so the
area formula applies to the (not necessarily smooth) bounded density of
\(P_0^{N,m}\).  Equivalently, one may first smooth inside this tube and then
remove the smoothing.  The Jacobian satisfies
\[
 \frac{\dd}{\dd s}\log\det D\Phi_s^N(\bm x)
 =\operatorname{div}_{\bm x}\bm B^N(\Phi_s^N(\bm x)),
\]
whereas
\[
 \frac{\dd}{\dd s}\log\pi^{\otimes N}(\Phi_s^N(\bm x))
 =\sum_{i=1}^N\bpi(\Phi_s^N(\bm x)_i)
   \cdot B_i^N(\Phi_s^N(\bm x)).
\]
Combining these two formulas with \eqref{eq:divergence} gives
\begin{equation}\label{eq:localized-entropy}
 \Ent(P_t^{N,m}\mid\pi^{\otimes N})
 +2N\int_0^t\E_{P_s^{N,m}}\F_N^\pi\,\dd s
 =\Ent(P_0^{N,m}\mid\pi^{\otimes N}).
\end{equation}
This is the same calculation as in the smooth proof, now with every
integration performed on a region uniformly separated from the singular
set.

It remains to pass \(m\to\infty\).  The local positivity of \(G_\pi\) and
compactness away from the diagonal give a constant \(C_G\) such that
\(G_\pi(x,y)\ge-C_G\) for \(x\ne y\).  Hence the elementary shift
\[
 \widehat{\mathcal E}_N^\pi:=\F_N^\pi+\frac{C_G}{2}
\]
is nonnegative for every configuration; the sharper Onsager estimate is not
needed for this fixed-\(N\) approximation.  Equation
\eqref{eq:localized-entropy} becomes
\begin{equation}\label{eq:localized-renormalized-entropy}
 \Ent(P_t^{N,m}\mid\pi^{\otimes N})
 +2N\int_0^t\E_{P_s^{N,m}}\widehat{\mathcal E}_N^\pi\,\dd s
 =\Ent(P_0^{N,m}\mid\pi^{\otimes N})
  +NC_Gt.
\end{equation}
Because \(\alpha_m\uparrow1\), conditioning on \(A_m\) gives
\[
 \Ent(P_0^{N,m}\mid\pi^{\otimes N})
 =\frac1{\alpha_m}\int_{A_m}h_0\log h_0\,\dd\pi^{\otimes N}
  -\log\alpha_m
 \longrightarrow \Ent(P_0^N\mid\pi^{\otimes N}).
\]
The right-hand side of \eqref{eq:localized-renormalized-entropy} is therefore
bounded.  Lower semicontinuity first shows that
\(P_t^N=(\Phi_t^N)_\#P_0^N\) has finite entropy.  Since the flow is injective,
\(P_t^{N,m}\) is precisely \(P_t^N\) conditioned on
\(\Phi_t^N(A_m)\).  The same conditional-entropy formula at time \(t\) now
gives
\[
 \Ent(P_t^{N,m}\mid\pi^{\otimes N})
 \longrightarrow\Ent(P_t^N\mid\pi^{\otimes N}).
\]
Finally, Tonelli's theorem and monotone convergence apply to the energy:
\[
 \int_0^t\E_{P_s^{N,m}}\widehat{\mathcal E}_N^\pi\,\dd s
 =\frac1{\alpha_m}\int_{A_m}\int_0^t
  \widehat{\mathcal E}_N^\pi(\Phi_s^N(\bm x))\,\dd s\,\dd P_0^N(\bm x)
 \longrightarrow
 \int_0^t\E_{P_s^N}\widehat{\mathcal E}_N^\pi\,\dd s.
\]
Passing to the limit in \eqref{eq:localized-renormalized-entropy} and
subtracting the deterministic correction
\(NC_Gt\) proves the exact identity
\eqref{eq:entropy-identity}.

\bibliographystyle{abbrvnat}
\bibliography{references}

\end{document}